\documentclass[12pt]{amsart}

\usepackage{graphicx}

\usepackage{color}

\newcommand{\NPHI}{\ncal_{\phi_{\lambda}}}

\newcommand{\gfrak}{{\mathfrak G}}

\newcommand{\VT}{\overline{V}_{T,\epsilon}}

\newcommand{\R}{{\mathbb R}}
\newcommand{\C}{{\mathbb C}}
\newcommand{\Q}{{\mathbb Q}}
\newcommand{\Z}{{\mathbb Z}}

\newcommand{\dbar}{\bar\partial}
\newcommand{\ddbar}{\partial\dbar}

\renewcommand{\phi}{\varphi}
\newcommand{\acal}{\mathcal{A}}

\newcommand{\ccal}{\mathcal{C}}

\newcommand{\fcal}{\mathcal{F}}

\newcommand{\hcal}{\mathcal{H}}
\newcommand{\ical}{\mathcal{I}}

\newcommand{\lcal}{\mathcal{L}}

\newcommand{\ncal}{\mathcal{N}}

\newcommand{\pcal}{\mathcal{P}}

\newcommand{\rcal}{\mathcal{R}}

\newcommand{\scal}{\mathcal{S}}

\newcommand{\ep}{\varepsilon}

\newcommand{\half}{{\frac{1}{2}}}

\renewcommand{\phi}{\varphi}

\newtheorem{theo}{{\sc Theorem}}[section]
\newtheorem{cor}[theo]{{\sc Corollary}}
\newtheorem{defn}[theo]{{\sc Definition}}
\newtheorem{defin}[theo]{{\sc Definition}}
\newtheorem{rem}[theo]{{\sc Remark}}

\newtheorem{lem}[theo]{{\sc Lemma}}

\newtheorem{prop}[theo]{{\sc Proposition}}

\title{Nodal intersections and geometric control}

\author{John A. Toth}
\address{Department of Mathematics and Statistics, McGill University, Montreal, CANADA}
\email{jtoth@math.mcgill.ca}

\author{Steve Zelditch }
\address{Department of Mathematics, Northwestern  University, Evanston, IL 60208, USA}
\email{zelditch@math.northwestern.edu}

\thanks{Research of J.T. was partially supported by NSERC Discovery Grant \# OGP0170280, an FRQNT Team Grant and the French National Research Agency project Gerasic-ANR-13-BS01-0007-0.
Research of S.Z. was partially supported by NSF grant  \# DMS-1541126   }

\begin{document}

\maketitle

\begin{abstract}   We prove that the number of nodal points on an {\it $\scal$-good} real
analytic curve $\ccal$  of a sequence $\scal$ of  Laplace
eigenfunctions $\phi_j$ of eigenvalue $-\lambda_j^2$ of a real analytic Riemannian manifold
$(M, g)$  is bounded above
by $A_{g, \ccal} \; \lambda_j$.  Moreover, we prove that the codimension-two Hausdorff measure $\hcal^{m-2}(\NPHI \cap H)$ of nodal intersections with a  connected, irreducible real analytic hypersurface $H \subset M$ is $\leq A_{g, H} \; \lambda_j$.The $\scal$-goodness condition is that  the sequence
of normalized logarithms $\frac{1}{\lambda_j} \log |\phi_j|^2$
does not tend to $-\infty$ uniformly on $\ccal$, resp. $H$.  We further show that a hypersurface satisfying a geometric control condition is  $\scal$-good for a density one subsequence of eigenfunctions.


\end{abstract}
\bigskip

This article is concerned with the growth of the number $n(\phi_{\lambda}, \ccal)$  of zeros of a sequence $\scal = \{\phi_{\lambda_j}\}_{j =1}^{\infty}$ of Laplace eigenfunction  $\phi_{\lambda_j}$ of eigenvalue $-\lambda_j^2$
on a  connected, irreducible real analytic curve $\ccal$ of a real analytic Riemannian
manifold $(M^m, g)$ of dimension $m$ without boundary.  To rule out degenerate cases,
we  assume (as in \cite{TZ}) that the pair
$(\ccal, \scal)$ satisfies a quantitative unique continuation condition  $ \|\phi_{j} |_{\ccal}\|_{L^{2}(\ccal)}
     \geq e^{- a \lambda_j} $  called  $\scal$- {\it goodness}. 
 (Definition \ref{DEFINTRO}).  When $\ccal$ is $\scal$-good, Theorem \ref{INTER}  asserts that there exists a constant
$A$ depending only on $g, \ccal$ so that \begin{equation} \label{UBintro} n(\phi_{\lambda_j}, \ccal)
\leq A \;\lambda_j, \;\; (\lambda_j \in \scal) \end{equation}  (see Figure 1). This bound generalizes Theorem 6 of \cite{TZ}  for Dirichlet/Neumann eigenfunctions of
 piecewise real analytic plane domains to any real analytic  Riemannian
manifold without boundary (of any dimension).  Motivation to study nodal points on curves and related results are discussed in Section \ref{OPEN}.
It is a  special case of estimating the codimension-two Hausdorff measure $\hcal^{m-2}(\NPHI \cap H)$ of nodal intersections with a  connected, irreducible real analytic hypersurface $H \subset M$ and in Theorem
\ref{NODALBOUND} we prove this generalization.

The main `defect' in Theorems \ref{INTER}-\ref{NODALBOUND} is that the condition that $(\ccal,  \scal)$ be $\scal$-good is subtle and difficult to establish. Much of  this article is devoted to providing sufficient conditions for `goodness'. The definition of $\scal$-good makes sense for any connected, irreducible analytic submanifold $H \subset M$, not only curves. One of the main results of this article (Theorems \ref{MASSMICRO}-\ref{mainthm1})   gives  a kind of geometric control condition that a $C^{\infty}$ hypersurface $H \subset M$ be $\scal$-good for a density one subsequence of an orthonormal basis of eigenfunctions. When $\dim M = 2$, the condition applies to curves and gives concrete and purely dynamical  conditions under which  \eqref{UBintro}  holds for a density one subsequence of eigenfunctions  (Theorem \ref{INTER2}).

To state our results, we need some notation. 
 We  denote by
$\{\phi_{j}\}_{j = 0}^{\infty}$  an orthonormal basis of Laplace eigenfunctions,
$$-\Delta \phi_{j} = \lambda_j^2 \phi_{j}, \;\;\; \langle \phi_j, \phi_k \rangle = \delta_{jk}, $$
where $\lambda_0 = 0 < \lambda_1 \leq \lambda_2 \leq \cdots$ and
where  $\langle u, v \rangle = \int_M u v dV_g$ ($dV_g$ being the
volume form). We denote a subsequence $\{j_k\}_{k = 1}^{\infty} $ of (indices of) eigenvalues
by $\scal$. By a slight abuse of notation, we also let $\scal$ denote the associated sequence $\{\lambda_{j_k}\}$ of eigenvalues or the sequence  $\{\phi_{j_k}\}$ of eigenfunctions from the given orthonormal basis.

Let $H\subset M$ be a connected, irreducible analytic submanifold. The assumptions that $H$ is connected, irreducible and analytic will be made
throughout the paper. 
Given a submanifold $H \subset M$, we denote the restriction operator to $H$ by $\gamma_H f = f |_H$. To simplify notation, we also write
$\gamma_H f = f^H$.
The criterion that a pair $(H, \scal)$ be good is stated in terms of the associated
 sequence 
  \begin{equation} \label{ujdef} u_j: =  \frac{1}{\lambda_j} \log |\phi_j |^2 \end{equation} of  normalized logarithms, and in particular their
restrictions  \begin{equation} \label{ujCdef} u_j^{H} : = \gamma_{H} u_j: =  \frac{1}{\lambda_j} \log |\phi_j^{H}|^2 \end{equation}
to $H$. We only consider the goodness of connected, irreducible, real analytic submanifolds.

 \begin{defn}\label{DEFINTRO}
Given a subsequence $\scal: = \{\phi_{j_k}\}$, we say
that a connected, irreducible real analytic submanifold $H \subset M$ is {\it $\scal$-good}, or that
$(H, \scal)$ is a good pair,  if the sequence  \eqref{ujCdef} with $j_k \in \scal$ does
{\bf not} tend to $-\infty$ uniformly on compact subsets of  $H$,
 i.e.  there exists
a constant $M_{\scal} > 0$ so that  $$ \;\;\;\sup_H u_j^{H} \geq - M_{\scal}, \;\; \forall j \in \scal 
.$$ 
 If $H$ is $\scal$-good when $\scal$ is the entire   orthonormal basis sequence, we say that $H$ is {\it completely good}.
   \end{defn}
   The opposite of a good pair $(H, \scal)$ is a bad pair. 
The terminology is not ideal, but was introduced in \cite{TZ} and used
in a number of articles (e.g. \cite{JJ, BR}) and so we continue to use it here.
Note that the connected, irreducible assumption is made to prohibit taking
unions $H_1 \cup H_2$ of two analytic submanifolds, one of which may be good and the other bad. By the definition above, the union would be good but the nodal bounds could be false.

 We denote the nodal set of an eigenfunction $\phi_{\lambda}$
of eigenvalue $- \lambda^2$ by
$$\ncal_{\phi_{\lambda}}= \{x \in M:
\phi_{\lambda} (x) = 0\}. $$  Our first result is the following 

\begin{theo} \label{INTER} Suppose that $(M^m,g)$ is a real analytic Riemannian manifold
of dimension m without boundary and that $\ccal \subset M$ is connected, irreducible real analytic
curve. If $\ccal$ is $\scal$-good, then there exists a constant $A_{\scal, g}$ so that 
$$n(\phi_j, \ccal) : = \# \{\ccal \cap \ncal_{\phi_j}\} \leq A_{\scal, g} \; \lambda_j, \;\; j \in \scal.$$
\end{theo}

Section \ref{INTERSECT} is devoted to the proof of Theorem \ref{INTER}.
As in \cite{TZ,Zint}, we prove the bound on nodal points on curves of Theorem \ref{INTER} by analytic
continuation of the eigenfunctions and curves to the complexification of $M$.
 Complexification is useful for upper bounds since the number $n(\phi_{\lambda}^{\C}, \ccal_{\C})$  of zeros of the complexified
eigenfunction on the complexified curve is $\geq$ the number of real zeros, 
i.e.
\begin{equation} \label{MORE} n(\phi_{\lambda}^{\C}, \ccal_{\C}): = 
\#\{\ncal_{\lambda}^{\C} \cap \ccal_{\C} \} \geq n(\phi_{\lambda}, \ccal):= \#\{\ncal_{\lambda}^{\R} \cap \ccal \}.
\end{equation}The same
technique was used in \cite{Zint} to obtain lower bounds on the number
of intersections of geodesics with the nodal set when  the geodesic flow is ergodic. Since there is a significant overlap with \cite{Zint,ZSt}, we refer to those
articles for much of the backround on complexification. 
The special case of Theorem \ref{INTER} where $M$ is a surface and $H$ is a $C^{\omega}$-curve was proved in \cite{CT} using a somewhat different frequency function approach.

\begin{center}
\includegraphics[width=6cm]{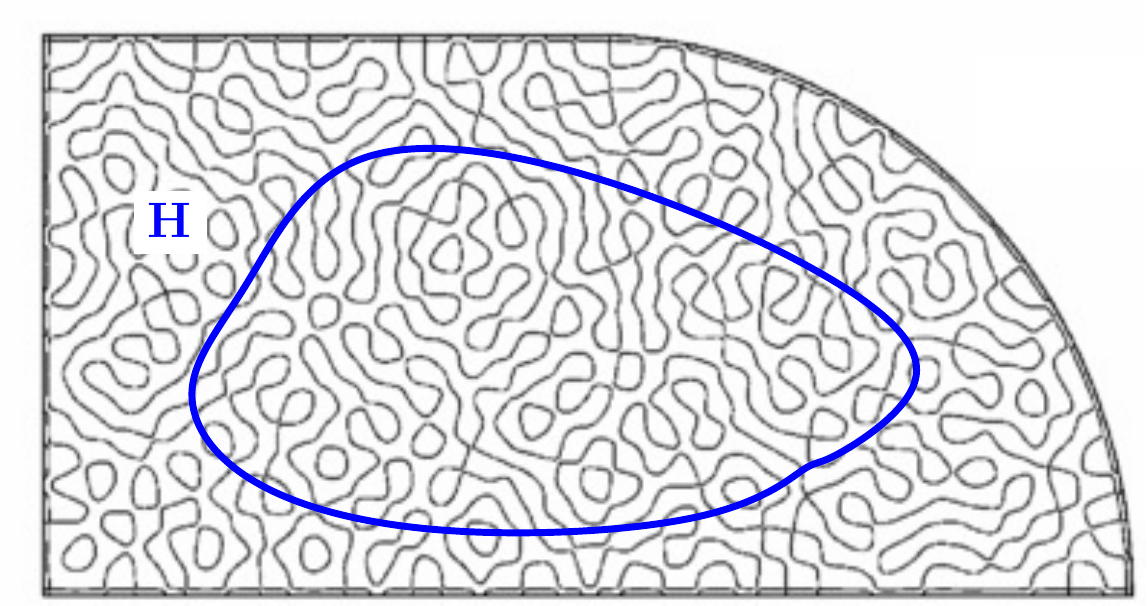}\ \\
\footnotesize{Figure 1: Nodal lines of a high energy state, $\lambda \sim 84$, in the quarter stadium.}
\end{center}

     We then generalize the theorem to real analytic hypersurfaces $H \subset M^m$ for manifolds of any dimension $m$. We separate out the statements and proofs because a new integral geometric method adapted from \cite{ZSt} is used in higher dimensions.

\begin{theo} \label{NODALBOUND} Let $(M^m, g)$ be a real analytic Riemannian manifold of dimension $m$ and let $H \subset M$ be a connected, irreducible, $\scal$-good real analytic hyperurface.
 Then, there exists a constant $C > 0$ depending only on
 $(M, g, H)$ so that
$$\hcal^{m-2} (\ncal_{\phi_{j_k}}  \cap H)  \leq C \lambda_{j_k}, \;\;(j_k \in \scal).$$

\end{theo}

The remainder
of the Introduction is concerned with criteria for goodness.

\subsection{\label{GDSECT} Measures of goodness }

There are some natural parameters associated with a good pair $(H, \scal)$. 
The first is the density of $\scal$. 
We recall that  the density of a set $\scal \subset  \mathbb{N}$ is defined by by
\[D^*(\scal):=
\lim_{X\to \infty } \frac{1}{X}|\{j\in \scal~|~ j<X\}|,
\] when the limit exists.  
When the limit does not exist we refer to the $\limsup$ as the upper density and the
$\liminf$ as the lower density.  We say ``almost all" when $D^*(\scal) = 1$
and if $(H, \scal)$ is a good pair with $D^*(\scal) = 1$ then we say that
$H$ is `almost completely good'.

The second natural parameter is the rate of decay of $||\phi_j^H||$ in the $L^2$-norm or sup-norm. 
  In  \cite{TZ,ET},   a  curve or other submanifold was defined to be  {\it good}  if
 there exists a constant $a > 0$ so that  for all $\lambda_j$ sufficiently large,
 \begin{equation}
\label{OLDGOOD} \|\phi_{j}^H\|_{L^{2}(H)}
     \geq e^{- a \lambda_j} .\end{equation}
      In \cite{ET} a `revised goodness' condition was defined by the  apriori stronger criterion that $\|\phi_{j}^H\|_{L^{\infty}(H)}
     \geq e^{- a \lambda_j} $.
     In \S \ref{EQUIVSECT} we show that \eqref{OLDGOOD} (and the sup-norm analogue)  are equivalent to
     Definition \ref{DEFINTRO}.  
  
A much stronger quantitative goodness condition is a uniform lower bound
$||\phi_{j_k}^H||_{L^2(H)} \geq C_{\scal}$ for the $L^2$-norms of restricted eigenfunctions in the sequence $\scal$. Somewhat surprisingly, our main criterion for goodness produces subsequences of density $\geq 1 - \delta$
for any $\delta > 0$ which  possess uniform lower bounds $C_{\delta}>0$.

\subsection{\label{MICROGOOD} A sufficient microlocal condition for goodness of a hypersurface} In this section, we give our main criterion for almost complete goodness of a hypersurface in the strong sense that the restrictions possess uniform lower bounds in the sense just mentioned. The criterion consists of two conditions on $H$: (i) asymmetry 
with respect to geodesic flow, and (ii) a full measure flowout condition.

  We begin with (i).    In \cite{TZ2}, a geodesic asymmetry condition on a hypersurface  was introduced
     which is sufficient that restrictions of quantum ergodic eigenfunctions on $M$  remain quantum ergodic on the hypersurface. It is reviewed in Definition \ref{ANC} and is the same as
      Definition 1 of \cite{TZ2} as well as  \cite{TZ3,DZ}. 
  It turns out that
the same asymmetry condition plus a flow-out condition implies that   a
     hypersurface is good for a density one  subsequence  of eigenfunctions  and that for any $\delta> 0$, the $L^2$ norms of the restricted eigenfunctions
     have a uniform lower bound $C_{\delta} > 0$ for a subsequence of density $1 - \delta$.   The asymmetry condition pertains to the
 two `sides' of $H$, i.e. to the two lifts of $(y, \eta) \in B^* H$ to unit 
 covectors $\xi_{\pm}(y, \eta) \in S^*_H M$ to $M$. We denote the symplectic volume measure on $B^* H$ by $\mu_H$. We define the symmetric
 subset $B_S^* H$ to be the set of $(y, \eta) \in B^*H$ so that 
 $G^t(\xi_+(y, \eta)) = G^t(\xi_-(y, \eta))$ for some $t \not= 0$.
     
     \begin{defin} \label{MADEF} $H$ is microlocally asymmetric if $\mu_H(B_S^*H) = 0$. \end{defin}

                 Next we turn to the  flow-out condition (ii). It  is that \begin{align} \label{ASSUME}
\mu_L ( \rm{FL}(H) ) = 1. \end{align} where 
      \begin{equation} \label{FLOWOUT} {\rm{FL}(H)}: =   \bigcup_{ t \in \R} G^t ( S_H^*M \setminus S^*H) \;\;  \end{equation}   
      is the geodesic  flowout of 
of the non-tangential unit cotangent vectors $S^*_H M \setminus S^*H$ along $H$. 
 Since $H$ is a hypersurface, $S^*_H M \subset S^* M$ is also a hypersurface which is almost everywhere transverse to the geodesic flow, i.e. it is a symplectic transversal (see \cite{TZ2}). It follows that
the flowout is an invariant set of positive measure in $S^*M$. When $G^t$ is ergodic on $S^* M,$  since $\rm{FL}(H)$ is $G^t$-invariant with $\mu_L(\rm{FL}(H)) >0,$   it follows that  every hypersurface satisfies \eqref{ASSUME},  but
 we do not assume ergodicity here.   In section \ref{examples}, we show that a large class of curves satisfy (\ref{ASSUME}) on  surfaces with completely  integrable geodesic flows. These include convex surfaces of revolution and Liouville tori satisfying generic twist assumptions. 

 The next result is a sufficient condition that $H$ be almost completely good.

\begin{theo} \label{MASSMICRO} Suppose that $H$ is a  microlocally asymmetric hypersurface satisfying \eqref{ASSUME}.

Then: if  $\scal = \{\phi_{j_k}\}$ is a sequence of eigenfunctions
satisfying $ ||\phi_{j_k} |_H||_{L^2(H)} = o(1)$, then the upper density
$D^*(\scal)$  equals zero. \end{theo}

 The following theorem gives a more quantitative version:

\begin{theo} \label{mainthm1}
Let $H \subset M$ be a microlocally asymmetric hypersurface satisfying \eqref{ASSUME}. 
Then, for any $\delta >0,$ there exists a subset $\scal(\delta) \subset \{1,...,\lambda \}$ of density $D^*(\scal(\delta)) \geq 1-\delta$ such that 
$$ \| \phi_{\lambda_j} \|_{L^2(H)} \geq C(\delta) >0,\quad j \in \scal(\delta).$$
\end{theo}




As mentioned above, the  assumption $ ||\phi_{j_k} |_H||_{L^2(H)} = o( 1)$ is  much 
weaker than the $\scal$- badness of $H$.
In fact, we do not know any microlocal (or other techniques) that prove
goodness without proving the stronger positive lower bound. There
do exist other non-microlocal  techniques which directly prove goodness.  In
 \cite{JJ}, J.  Jung
     proved that geodesic distance circles and horocycles in the hyperbolic plane are good relative to eigenfunctions on compact or finite area hyperbolic surfaces.  In \cite{ET} it is proved that curves of positive geodesic curvature  are good for Neumann or Dirichlet quantum ergodic eigenfunctions on a Euclidean plane domain.

          \subsection{The main results on counting nodal points on curves or measuring Hausdorff measures  on hypersurfaces}
A combination of Theorems \ref{INTER} and \ref{mainthm1} gives the main result on nodal intersections:

\begin{theo} \label{INTER2}
Let $\ccal $ be an asymmetric   $C^{\omega}$ curve on a compact, closed, $C^{\omega}$ Riemannian surface $(M^2,g)$ satisfying \eqref{ASSUME}.  Then, for any $\delta >0$ there exists a   subsequence $\scal(\delta)$ with $D^*(\scal(\delta)) \geq 1-\delta$ for which $\ccal$ is $\scal'$-good and a constant $A_{\scal,g}(\delta)>0$ such that
$$ n(\phi_j, \ccal) : = \# \{\ccal \cap \ncal_{\phi_j}\} \leq A_{\scal, g}(\delta) \; \lambda_j, \,\,\,\,\, j \in \scal(\delta).$$
\end{theo}\

  The higher dimensional generalization is as follows:
  \begin{theo} \label{NODALBOUND2}
Let $H $ be an asymmetric   $C^{\omega}$ hypersurface  of a compact, closed, $C^{\omega}$ Riemannian manifold $(M^m,g)$ satisfying \eqref{ASSUME}.  Then, for any $\delta >0$ there exists a   subsequence $\scal(\delta)$ with $D^*(\scal(\delta)) \geq 1-\delta$ for which $\ccal$ is $\scal'$-good and a constant $A_{\scal,g}(\delta)>0$ such that
 
$$\hcal^{m-2} (\NPHI \cap H)  \leq A_{\scal, g}(\delta) \; \lambda_j, \,\,\,\,\, j \in \scal(\delta).$$

\end{theo}\

\subsection{Relating weak* limits on $M$ and on $H$}

The rest of the article is devoted to proving Theorems \ref{MASSMICRO}-
 \ref{mainthm1}, which together with Theorem \ref{INTER} imply  Theorem \ref{INTER2}. These results belong to the theory of weak* limits and geometric control theory and seem to us to have an independent interest.
 
 We recall that an invariant measure $d\mu$ for the geodesic flow on $S^*M$ is
called a microlocal defect (or defect measure, or quantum limit) if there exists a sequence $\{\phi_{j_k}\}$ of eigenfunctions
such that $\langle A \phi_{j_k}, \phi_{j_k} \rangle_{L^2(M)} \to \int_{S^*M} \sigma_A d\mu$ for all pseudo-differential operators
$A \in \Psi^0(M)$. There are analogous notions for semi-classical pseudo-differential operators. We assume familiarity with
these notions and refer to \cite{Zw} for background.

In Section \ref{ME} we relate matrix elements of eigenfunctions on $M$ to
those of their restrictions to a hypersurface $H$. This material is largely drawn from \cite{TZ2}, and we review the necessary background in Section \ref{BACKGROUND}. There is an obvious relation between matrix elements on $M$ and matrix elements on $H$ given in Lemma \ref{COMPLEM}. It involves a time average $\VT(a)$ of $\gamma_H^* Op_h(a) \gamma_H$.
In \cite{TZ2}, $\VT(a)$ was decomposed into a pseudo-differential term $P_{T, \epsilon}$ and a Fourier integral term $F_{T, \epsilon}$ (see Proposition \ref{VTDECOMPa}). The symbol of  $P_{T, \epsilon}$ is essentially
a flow-out of $a$ using that $S^*_H M$ is a sort-of cross-section to the geodesic flow.\footnote{As discussed below, it is not even literally a cross
section of $FL(H)$.}  It was proved in \cite{TZ2} (see also \cite{DZ}) that for asymmetric hypersurfaces,  the matrix elements of
$F_{T, \epsilon}$ tend to zero almost surely. For the sake of completeness,  sketch the proof in   Section \ref{FTSECT}  that for any $(T, \epsilon)$ there exists a subsequence $\scal_F$ of density one so that the matrix
elements  $\langle F_{T, \epsilon} \phi_{j_k}, \phi_{j_k} \rangle_{j_k \in \scal_F} \to 0$.   We exploit this fact in the following Proposition, which 
relates microlocal defect measures (quantum limits) of the eigenfunctions and of their restrictions to $H$.  To state the result 
precisely, we need some further notation.
For fixed $\epsilon \in (0,1),$ let $\chi_{\epsilon}^H \in C^{\infty}_0(B^*H)$ be a cutoff with supp $\chi_{\epsilon}^H \subset \{ (s,\sigma) \in B^*H:  | |\sigma|_s - 1 | < \epsilon \}.$

\begin{prop} \label{FTPROP} Suppose that $H$ is asymmetric. Then, for any $T,\epsilon >0$ there exists a density-one sequence $\scal_F(T,\epsilon)$ such that   for $a \in S^0(H),$ $$\lim_{k \to \infty; \, j_k \in \scal_F(T,\epsilon)}  \Big( \langle Op_H(a(1-\chi_{\epsilon}^H) ) \phi_{j_k} |_H, \phi_{j_k} |_H \rangle_{L^2(H)}  - \langle P_{T, \epsilon}(a) \phi_{j_k}, \phi_{j_k} \rangle_{L^2(M)} \Big) = 0. $$
\end{prop} 
To simplify notation, in the following we will simply write $\scal_F:= \scal_F(T,\epsilon)$ suppressing the dependence on $T,\epsilon>0.$
It is necessary in general to remove a density zero subsequence. For instance, special sequences of Gaussian beams along a geodesic $\gamma$   blow up when restricted to $\gamma$. 


The following Theorem asserts that the microlocal defect measures  on $S^*M$ of  typical subsequences on $M$ induce finite measures on $S^*_H M$ and $B^*H$. This cannot be true for all subsequences in general, since restrictions of   subsequences of  eigenfunctions to hypersurfaces can blow up in the $L^2$ norm. This happens for instance in the case of highest-weight spherical harmonics $\phi_{k}(x,y,z) = c_0 \, k^{\frac{1}{4}} \, (x+ iy)^k; \, k=1,2,3,...$ with $(x,y,z) \in S^2$ and $H = \{ (x,y,z) \in S^2; z =0 \}.$  

\begin{theo} \label{MDMPROP} 
Suppose that $H$ is a  microlocally asymmetric hypersurface.  Then, there exists a density-one  subsequence $\tilde{\scal}$ with the property that to any microlocal defect measure $d\mu$ of a subsequence $\scal \subset \tilde{\scal} $ there corresponds 
a `disintegration measure' $d\mu_{\scal}^H$ on $B^* H$ such that  
$$\langle Op_H(a) \phi_{j_k}|_H, \phi_{j_k}|_H \rangle \to \int_{B^* H} a \, d\mu_{\scal}^H, \quad a \in S^0(H). $$
\end{theo}

Since $B^*H$ is diffeomorphic to $S^{*, \pm}_H M$, one can rewrite the integral in Theorem \ref{MDMPROP} over $S_{H}^*M$ instead of $B^*H.$ The QER theorem of  \cite{TZ} is  the special case where $\mu = \mu_L$ (Liouville measure) and the geodesic flow $G^t: S^*M \to S^*M$ is ergodic
with respect to $\mu_L$. 





The definition of $d \mu_{\scal}^H$ is given in Section \ref{DEFMES} and is essentially the relation between a flow-invariant measure on $S^*M$ and its
disintegration in terms of  an induced invariant measure on the cross section $S^*_H M$. But as explained in Section \ref{DISSECT}, $S^*_H M$ is not a genuine cross-section and one cannot always express the disintegration measure as a measure on $S^*_H M$. This obstruction is responsible for the possible deletion of a zero density subsequence. We mainly use Theorem \ref{MDMPROP} in the case where $d\mu_{\scal}^H = 0$, which forces $d\mu_{\scal}^M = 0$. This can be compared with the possible microlocal defect measures of $\scal$ on $S^*M$, showing that they must have zero integrals against $\sigma_{P_{T, \epsilon}}$.

\begin{rem} It would be interesting to see if the hypotheses of Theorem \ref{MASSMICRO} (and the related results on weak* limits of 
restrictions) can be weakened, and if the conclusion can be strengthened. For instance, one `loss' of a density zero subsequence occurs
in  Lemma \ref{FTOUT}. But  it is possible that $\langle F \phi_{j_k}, \phi_{j_k} \rangle $ tends to zero for the entire sequence. 
It is then  possible that the hypotheses imply $H$ is $\scal$-good
for the entire sequence of $\phi_j$. It is also possible that asymmetry alone is a sufficient hypothesis for the density one statement. \end{rem}

\subsection{Pluri-subharmonic theory and goodness}

It is natural to ask if the theory of PSH (pluri-subharmonic) functions can help identify good curves.  As mentioned above, `goodness' is a much weaker condition than possession of uniform lower $L^2$ bounds.  In Section \ref{GOODSECT} we draw some rather modest conclusions from the literature of PSH functions. The weakness of the conclusions is due to the fact that they are valid for general sequences of
PSH functions and do not make full use of the assumption that our sequences 
are log-moduli of eigenfunctions \eqref{ujdef}. What seems to be lacking is a theory of  $L^1$ limits of normalized log-moduli of complexified eigenfunctions (\eqref{ujdef} or  \eqref{ujCdef}).  For instance, no connection is known relating  such limits to the geodesic flow. Developing a microlocal theory of such limits seems to us a fundamental problem.

 Except in Sections \ref{INTERSECT} and \ref{GOODSECT} we do not employ complex analytic methods.

\subsection{\label{OPEN}Related results and open problems}

There are several motivations to study nodal points on curves.  Nodal sets and curves have complementary dimension, so that the number of intersections is finite under a suitable transversality hypothesis. The goodness assumption gives a strong formulation of this transversality. 

  One motivation is that Crofton's formula
 expresses the  Hausdorff measure of a hypersurface $Y$
sets as the average number  of intersections   of $Y \subset M^m$ with a random
line (or geodesic arc). When $Y = \ncal_{\phi_{\lambda}}$ is a nodal hypersurface, this method was used in \cite{DF} to obtain upper bounds on $\hcal^{m-1}(\ncal_{\phi_{\lambda}})$. More precisely, Crofton's formula implies that $$\hcal^{m-1}(\ncal_{\phi_{\lambda}}) \leq \int_{\lcal} \#\{L \cap    \ncal_{\phi_{\lambda}}\} d\mu(L)$$
where $\lcal$ is the set of unit geodesic arcs and $d\mu$ is the Crofton measure \cite[p. 164]{DF} (and \cite[p. 178]{DF}.  As explained there, for  polynomials  of degree $\lambda$, 
$\#\{L \cap  \ncal_{P_{\lambda}}\} \leq C \lambda$ $d\mu$-almost everywhere, and a more complicated argument establishes the integral
bound for eigenfunctions.  A related argument is given in  \cite[Lemma 3.2]{Lin} using Crofton's formula \cite[(3.21)]{Lin} and
an upper  bound on the number of zeros of a non-zero analytic function in the unit disc in terms of its frequency function.  In  \cite{ZSt}, the analogous sharp  upper bound for nodal sets of Steklov eigenfunctions was  proved using Crofton's formula.  
The potential theoretic facts of Section \ref{GOODSECT} show that  $\#\{L \cap  \ncal_{\phi_{\lambda}}\} \leq C \lambda$ $d\mu$-almost everywhere for eigenfunctions.

Counting zeros on curves is also the mechanism for obtaining lower bounds 
on numbers of nodal domains on certain surfaces (see e.g. \cite{GRS,JJZ}).  In contrast to this article, the main point is to obtain lower bounds on numbers of nodal points on special curves rather than upper bounds.

Another question raised and studied  by Bourgain-Rudnick \cite{BR} is
to characterize the possible submanifolds $Y$ on which some sequence $\scal$ of eigenfunctions vanishes.  In our language,  $Y$ is {\it nodal} (Definition \ref{NODALCURVE}), which is an extreme form of $\scal$-bad. Theorem \ref{MASSMICRO} shows that $D^*(\scal) = 0$ if $Y$ is an  asymmetric hypersurface satisfying \eqref{ASSUME}.  This is non-trivial, since the standard example of odd eigenfunctions vanishing on the fixed point set of an isometric involution shows that a positive density sequence can vanish on a hypersurface. But the results of this paper  do not determine whether there exists a subsequence of density zero vanishing on such a hypersurface. The Bourgain-Rudnick question can be generalized as follows:  \bigskip

     \noindent{\bf Problem}  Characterize submanifolds $H$ which are $\scal$-bad for some subsequence. Moreover, characterize $H$ which are bad for a  positive   density subsequence  $\scal = \{\phi_{j_k}\}$ of eigenfunctions, that is, 
 $ ||\phi_{j_k} |_H||_{L^2(H)} \leq e^{- M \lambda_{j_k}}$ for
all $M$.   Must the sequence actually vanish on $H$? \bigskip

On a flat torus all periodic geodesics are $\scal$-bad; in fact, they are nodal in the sense of Definition \ref{NODALCURVE} (see subsection \ref{nodal curves}). On the other hand, if $H \subset \R^2 / (2\pi \Z)^2$ is a strictly convex curve, it is proved in \cite{BR} that 
\begin{equation} \label{brbound}
\| \phi_{\lambda} \|_{L^2(H)} \geq  C_H >0.\end{equation}
Consequently, any such curve $H$ is good. 
We note that if $H$ is strictly convex, it is not hard to show that $H$ is microlocally asymmetric in the sense of Definition \ref{MADEF} and also satisfies the flowout assumption $\mu_L( \rm{FL}(H)) =1.$ Consequently, the lower bound in (\ref{brbound}) is also a consequence of our Theorem \ref{mainthm1}, albeit only for an eigenfunction sequence of density arbitrarily close to one. 


The methods of this article and of \cite{TZ} are rather different, though both are based on analytic continuation. In this article we analytically continue the Poisson-wave kernel. At the present time, the analytic continuation is only known for manifolds without boundary (see \cite{ZPl,L,St}). The analytic continuation is based on parametrix constructions which are not known at present 
for general manifolds with boundary. This is obviously an interesting problem. 
Parametrices are known for diffractive (concave) boundaries, and that would be a natural first step.

    In \cite{TZ} we used the analytic continuation of  Euclidean
 layer potentials of $\R^2$ for bounded analytic domains as semi-classical Fourier integral operators. 
 This construction should generalize to all dimensions and also to 
 complete manifolds of negative curvature, where it is known that layer
 potentials are singular Fourier integral operators. The latter statement may hold in a suitable sense for domains in general complete Riemannian manifolds but to our knowledge this also remains an open problem.

   \subsection{Acknowledgements} We thank J. Galkowski for discussions 
   of our geometric control condition, and Z. Rudnick for discussions of his results with Bourgain on nodal curves and hypersurfaces.

\section{\label{GOODSECT} Good curves and submanifolds}

The definition of `goodness' in Definition 
\ref{DEFINTRO} is motivated by properties of sequences of subharmonic functions, and they are used in the proof of Theorem \ref{INTER}. The sequence $u_j$ is not subharmonic on $M$ but
has a natural extension to the complexification $M_{\C}$ of $M$ as 
subharmonic functions.  We denote the extension of \eqref{ujdef} by

  \begin{equation} \label{ujdefC} u_j^{\C}: =  \frac{1}{\lambda_j} \log |\phi_j^{\C} |^2, \end{equation} and their restrictions \eqref{ujCdef} to a complexified analytic
  submanifold $H_{\C}$ by 
  \begin{equation} \label{ujdefCC} u_j^{H,\C} : = \gamma_{H_{\C}} u_j^{\C}: =  \frac{1}{\lambda_j} \log |\phi_j^{\C} |_{H^{\C}}|^2. \end{equation}

 As we show in section \ref{EQUIVSECT},  the Definition \ref{DEFINTRO}  of `good' is equivalent to the following complex version:
 \begin{defn}\label{CXGOOD}
Given a subsequence $\scal: = \{\phi_{j_k}\}$ of eigenfunctions,  we say
that a connected, irreducible  real analytic submanifold $H \subset M$  is {\it $\scal$-good} if the sequence  \eqref{ujdefCC} with $j_k \in \scal$ does
{\bf not} tend to $-\infty$ uniformly on compact subsets of  $H_{\C}$,
 i.e.  there exists
a constant $M_{\scal} > 0$ so that  $$ \;\;\;\sup_{H_{\C}} u_j^{H,\C} \geq - M_{\scal}, \;\; \forall j \in \scal.$$ 
Otherwise we call $H$ $\scal$-bad.
  \end{defn}
  
  Here, $H_{\C}$ refers to some Grauert tube of $H$ in $M_{\C}$. The Definition does not depend on the specific radius, nor whether we use the 
intrisinc Grauert tube of $H$ or the intersection of $H_{\C}$ with a Grauert tube $M_{\epsilon}$ of $M$.

 Thus,  $H$ is $\scal$-bad if the $u_j^{H,\C} \to
-\infty$ unformly on compact subsets of $H$. 
   If $H$ fails to be good, then there  exists a sequence $\scal$ 
  so that $H$ is $\scal$-bad and we refer to $H$ as a {\it bad} sequence
  for $H$.  The simplest example of a bad pair $(H, \scal)$ is where
  the the eigenfunctions of $\scal$ vanish on $H$;  in this case we say that
  $H$ is a nodal  submanifold (see Definition \ref{NODALCURVE}.) Examples of nodal  hypersurfaces are fixed point sets of  an isometric involution, and then  $H$ is $\scal$-bad for the sequence of odd eigenfunctions

 It is also obvious that
if   a real analytic arc $\beta$, or piece of a real analytic submanifold $H$. is bad then the entire analytic
continuation of it $H$ is bad.

 These definitions are motivated by the  standard compactness Lemma 
 for families subharmonic functions (see  \cite{LG} or \cite{Ho2},Theorems 3.2.12-3.2.13).
 Let $v^*$ denote the USC (upper semi-continuous) regularization of $v$. 
\begin{lem} \label{HARTOGS} For any compact  connected irreducible analytic Riemannian manifold $(M, g)$,  and any real analytic submanifold $H$,
the family of  pluri-subharmonic functions \eqref{ujdefCC},
$$\fcal^{H}: = \{ u_j^{H,\C}, \; j = 1,  2, \dots \} $$ on $H_{\tau}$   is precompact in 
$L^1_{loc}(H_{\tau})$ as long as 
it does not converge  uniformly to $-
\infty$ on all compact subsets of $H_{\tau}$. Moreover:

\begin{itemize}

\item 
$\limsup_{k \to \infty} u_{k}^{H,\C}(t + i \tau)
 \leq 2 |\tau_H| $.

\item Let  $\{u_{j_k}^{H,\C}\}$ be any subsequence of  $\{u_j^{H,\C}\}$ with a unique $L^1_{loc}$ limit
$v$ on $S_{\epsilon}$ and let $v^*$ be its USC regularization.
 Then if  $v^* <  2 |\tau_H| - \epsilon$ on an open set $U \subset S_{\epsilon}$  then  $v^*  \leq 2 |\tau_H| - \epsilon$
for $\tilde{U} = \bigcup_{t \in \R} (U + t)$  and 
\begin{equation} \label{YBAD} \limsup_{k \to \infty} u_{j_k}^{H,\C}\leq |\tau_H| - \epsilon \;\;\; \mbox{on}\;\; \tilde{U}. \end{equation}

\end{itemize}

\end{lem}

The conditions of connectedness and irreducibilty arise from this Lemma. 
The original statement in \cite[Theorem 3.2.12]{Ho2} pertains to sequences
of subharmonic functions on connected open sets $U \subset \R^n$. Since the theorem is local it generalizes with no essential change to connected, irreducible complexified
hypersurfaces of $M_{\C}$. Clearly, connectedness is necessary: as mentioned in the introduction, unions $H_1 \cup H_2$ of a disjoint good and bad
hypersurface would be good. If $H_1 \cap H_2 \not= \emptyset$ then
$H_1 \cup H_2$ might be connected but $H_1 \cup H_2$ would still be good. 
The condition of `connected irreducible' means that $H$ has only one component. Hence in taking unions,  each hypersurface separately must be good.

\subsection{Bad submanifolds are polar}

In this section we review results on sequences of pluri-subharmonic functions which imply that $\scal$-bad sets are polar. This is not a restriction on real analytic curves or hypersurfaces (since they are necessarily polar) but is useful in proving the equivalence of different notions of `good'.

Let $\{u_j\}$ be the sequence \eqref{ujdef} of pluri-subharmonic functions.
Let $u \in L^1(M_{\tau})$. We say  that a subsequence $\{u_{j_k}\}_{j_k \in \scal}$ is a {\it u-sequence } if  $u_{j_k} \to u$ in $L^1(M_{\tau})$.

\begin{defn} Suppose that  $\{u_j\}_{j \in \scal}$
is a u-sequence. Define $$W_{\scal} =\{z \in M_{\tau}: \limsup_{j \to \infty} u_j(z) < u(z)\}. $$
\end{defn}

The following Proposition 1.39  from \cite{LG} (see also Theorem 1.27 of \cite{LG} and    Theorem 3.4.14 of \cite{Ho2})  will be
relevant:

  \begin{prop} \label{HORM} If $\scal = \{u_j\}$ is a sequence of pluri-subharmonic functions on an open
set $U$ and
$u_j \to u$ in $L^1(U)$ then the set of points $W_{\scal} $ in $U$ where
$$\limsup_{j \to \infty} u_j < u $$
is pluri-polar. \end{prop}

\begin{defn} 

Given a subsequence $\scal \subset {\mathbb N}$, we define
$\pcal_{\scal} \subset M_{\tau}$ be the set of points $z$  satisfying 
$$\limsup_{j \in \scal} u_j(z) = - \infty. $$

\end{defn}

Thus,  $\pcal_{\scal} \subset W_{\scal}$ and  $\pcal_{\scal}$ is  contained in a pluri-polar set. The   Hausdorff dimension 
 of a polar set in $\R^m$  is $\leq m - 2$ (\cite{Ho2}).  Since the statement is local the proof applies to $W \subset M_{\tau}$.

\subsection{\label{EQUIVSECT} Equivalence of different notions of goodness}
Here we prove  the equivalence of the following notions
of goodness for a real analytic function on a real analytic curve. \bigskip

\begin{enumerate}

\item Goodness in the sense of  Definition \ref{DEFINTRO}.\bigskip
     
     \item Goodness in the sense of Definition \ref{CXGOOD}.\bigskip

\item Goodness in the sense that $ \|\phi_{j} |_H\|_{L^{2}(H)}
     \geq e^{- a \lambda_j} $. \bigskip

     \item Goodness in the sense of $ \|\phi_{j} |_H\|_{L^{\infty}(H)}
     \geq e^{- a \lambda_j} $.

\end{enumerate}

Each {\it goodness} criterion implies that there is a point  $q \in H_{\epsilon}$ where $\lim_{j \to \infty} u_j^H(q) \geq - M $ for some $M > 0$.  Hence,
they all imply goodness in the sense (2) of   Definition \ref{CXGOOD}. The main
content of the equivalence is that the latter criterion (2) implies (1)-(3). This is non-obvious since these criterion only involve the behavior of $u_j$ on the real points of $H_{\C}$.

\begin{prop} If $H$ is a real analytic curve, then (1)-(4) are equivalent. 
\end{prop}

\begin{proof}

First, consider the simplest case where $H$ is a curve such that (2) holds. 
Then $\{u_j^H\}$ is pre-compact in $L^1$. Proposition \ref{HORM} then
implies that the set where $u_j \to - \infty$ is polar in $H_{\C}$. Since it 
has Hausdorff dimension 0 in $H_{\C}$, it cannot contains the real curve
$H$ and there must exist points such that (1) holds. In fact, such points of the real curve must have dimension 1 and so (3)-(4) also hold. 
That is,   for any  $\epsilon > 0$, there exists $M > 0$ and a measureable
subset $E \subset H$ of  $\hcal^1$-measure $\geq 1 - \epsilon$  where  $\frac{1}{\lambda_j} \log |\phi_j(z_0)| \geq -  M$
on $E$. But then $||\phi_{j}||_{L^1(H)} \geq ||\phi_{j}||_{L^1(E)}
\geq e^{- M \lambda_j} |E|.$\bigskip

\end{proof}

 Alternatively, one can prove the equivalences between $(1), (2)$ and $(3)$ using the following Hadamard three circles argument. We first treat the case where $\dim H = 1$ and $n=1.$

\begin{proof} $(3) \implies (2)$ since there must exist a point $q \in H$ at which $|\phi_{\lambda}(q)| \geq e^{-C \lambda}.$

$(2) \implies (3)$  Suppose

\begin{equation} \label{weak good}
\sup_{z \in H_{\C}} | \phi_{\lambda}^{\C}(z)| \geq e^{-C \lambda}. \end{equation}






Let $H, H_{\epsilon_1} = \{ z \in H_{\C}; \sqrt{\rho} = \epsilon_1\}$ and $H_{\epsilon_2} = \{ z \in H_{\C}; \sqrt{\rho} = \epsilon_2 \}$ with $0 < \epsilon_1 < \epsilon_2$ be three level curves  in the tube $H_{\C}.$ Without loss of generality, we also assume that  

$$ \sup_{z \in H_{\epsilon_1}} |\phi_{\lambda}^{\C}(z)| = e^{-C \lambda}.$$
By the Hadamard three circles theorem,  with $0 < \theta <1,$
\begin{align} \label{hadamard}
\sup_{z \in H_{\epsilon_1}} |\phi_{\lambda}^{\C}(z) | &\leq   \sup_{z \in H_{\epsilon_2}} |\phi_{\lambda}^{\C}(z) |^{1-\theta}  \times \sup_{q \in H} |\phi_{\lambda}(q)| ^{\theta} \nonumber \\
&\leq e^{2\epsilon_2 (1-\theta) \lambda} \cdot \| \phi_{\lambda} \|_{L^{\infty}(H)}^{\theta}. \end{align}

In the last line we needed a sup estimate for $|\phi_{\lambda}^{\C}|$. For this, we recall that \cite{ZPl}  
$$ \| \phi_{\lambda}^{\C} \|_{L^{\infty} (H_{\epsilon_2}) } = O( \lambda^{\frac{n-1}{2}} e^{ \epsilon_2 \lambda}) = O(e^{2 \epsilon_2 \lambda}).$$

Consequently, by the weak goodness assumption (\ref{weak good}) and (\ref{hadamard}),

$$ \| \phi_{\lambda} \|_{L^\infty(H)} \geq e^{-C \lambda}.$$

By continuity, we choose $q_0 \in H$ so that

$$|\phi_{\lambda}(q_0)| = e^{-C \lambda}.$$

Let $q: [0,L] \to H$ be the arclength parametrization with arclength parameter $s.$ By  the standard  bound for Laplace eigenfunctions, one also has that

\begin{equation} \label{poly}
\| \partial_s  \phi_{\lambda} \|_{L^\infty(H)} = O( \lambda^{n+1/2}). \end{equation}

Since by (\ref{poly}) the tangential derivative of $\phi_{\lambda}$ along $H$ has at most polynomial growth in $\lambda$, it follows by Taylor expansion along $H$ centered at $q_0$ that there is an subinterval $I(\lambda) \subset H$ containing $q_0$ of length $e^{-C' \lambda}$ with $C'>C>0$ such that for $q \in I(\lambda),$
$$ |\phi_{\lambda}(q) | \geq e^{-C'' \lambda}.$$
Consequently,

$$ \| \phi_{\lambda} \|_{L^2(H)} \geq e^{-C'' \lambda}$$
and so, $H$ is good in the sense of (2).

We note that the argument above using (\ref{poly}) also proves that $(2) \iff (3).$

 \end{proof}
 
The case where $H$ is a real analytic submanifold of dimension $\geq 2$ is more complicated because $H \subset H_{\C}$ has codimension $\geq 2$ and is not ruled out as a pluri-polar set. Rather we use that it is a totally real submanifold. The equivalence then follows from an (unpublished) theorem of B. Berndtsson, which says that if $H \subset \Omega$ is totally real submanifold of a complex manifold $\Omega$ and if $\{u_j\}$ is a sequence of pluri-subharmonic functions converging in $L^1(\Omega)$ to $u$, then
$u_j |_H \to u |_H$ in $L^1_{\rm{loc}}(H)$ \cite[Theorem 3.3]{Ber}. It follows
immediately that (2) implies (1) and (3).

\subsection{Nodal curves} \label{nodal curves}

The only known examples of bad curves are nodal curves in the following sense:

\begin{defn} \label{NODALCURVE} We say that a curve (e.g. a geodesic) $H$ is a nodal curve (geodesic) if there exists
a sequence $\{\phi_{j_k}\}$ of distinct  eigenfunctions which vanish in $H$.
Similarly for submanifolds of higher dimension.
\end{defn}

There are many examples of nodal geodesics.   These include: 

 \begin{itemize}
 
  \item  Rational radial geodesics on the unit disc or rational meridians on the unit
sphere are nodal geodesics. That is, one may fix an axis
 of rotation and consider real and imaginary parts of the associated
 basis $Y^m_{\ell}(\theta, \phi) $ of spherical harmonics to get $ \sin m \theta P_{\ell}^m(\cos \phi)$. 
 Here, $\frac{\partial}{\partial \theta}$ is the generator of the rotations. 
 Obviously the meridians defined by $\sin m \theta = 0$ i.e. $\theta = \frac{j \pi}{m}$  are nodal geodesics through the poles for the sequence with
 $m$ fixed and $\ell$ varying. Since $m$ is arbitrary,
 any `rational meridian' is a nodal geodesic, where rational means that
 the the angle to the fixed meridian $\theta = 0$ is a rational number
 $\frac{j \pi}{m}$ times $\pi$.  \bigskip

\item Fixed point sets of involutions
on surfaces of negative curvature are nodal geodesics.  
Thus, any closed geodesic of the standard $S^2$ is nodal with respect to its
 associated
 odd eigenfunctions. 
 
 \item Periodic geodesics on a flat torus $\R^2 / (2\pi \Z)^2.$ Given a periodic geodesic $\gamma(t) = ( m t, n t); \, 0\leq t \leq 2\pi$ with $(m,n) \in \Z^2,$ the sequence of Laplace eigenfunctions
 $$ \phi_k(x,y) = \sin \, ( k (nx-my) ), \quad (x,y) \in [0,2\pi] \times [0,2\pi], \,\, k=0,1,2,3,...$$
 clearly satisfies $\phi_k |_{\gamma} = 0$ and so $\gamma$ is nodal. In  \cite{BR2} (see Theorem 1.1), Bourgain and Rudnick prove that in fact segments of periodic geodesics are the only real-analytic nodal curves on a flat torus. In higher dimensions, they prove that postively-curved hypersurfaces on the flat torus cannot be nodal.
 
\end{itemize}

It is not clear
at present whether  a geodesic  $H \subset W$ that fails to be good is necessarily
a nodal geodesic. 
 Another question is whether bad nodal curves must be geodesics and
more general bad curves (if they exist) must be geodesics. In the case of
ergodic eigenfunctions, this question is studied in \cite{ET}. 
 Even in the case of the sphere,  the characterization of nodal curves seems a rather difficult and open problem. For instance,  it is unknown whether or not a circle of latitude different from the equator is nodal. The latter question is closely related to a classical conjecture of Stieltjes (see \cite{BR2}).

\section{\label{INTERSECT} Proof of Theorem \ref{INTER} }

The proof of Theorem \ref{INTER} is based on the analytic continuation
of eigenfunctions to a Grauert tube $M_{\tau}$ in the complexification
of $M$.   We will not review the background on Grauert tubes and on
analytic continuation of eigenfunctions and Poisson kernels, but refer 
to \cite{Zerg, ZPl, Zint, ZSt} for the necessary material. 

      We   recall that any real analytic
manifold $M$ admits a Bruhat-Whitney complexification $M_{\C}$,
and that for any real analytic metric $g$ all of the
eigenfunctions $\phi_j$  extend holomorphically to  a fixed open
open neighborhood $M_{\epsilon}$  of $M$ in $M_{\C}$ called a
Grauert tube of radius $\epsilon$. We recall that the square of the Grauert tube function is  $\rho(z) = \frac{1}{4} r^2(z, \bar{z})$
where $r^2$ is the analytic continuation of the distance-square function.
The Grauert tube of radius $\tau$ is denoted by $M_{\tau}$ and its
boundary $\partial M_{\tau}$   is the level set $\sqrt{\rho} = \tau$. 
 Given a real analytic hypersurface $H$, we 
define $H_{\epsilon} : = H_{\C} \cap M_{\epsilon}$.

  We denote the holomorphic extension of an eigenfunction
$\phi_{\lambda}$ by $\phi_{\lambda}^{\C}$, respectively elements
of an orthonormal basis by  by $\phi_j^{\C}$. The complex nodal set is denoted  by
\begin{equation} \label{ncalC} \ncal_{\phi_{\lambda}^{\C}}= \{z \in M_{\C}:
\phi_{\lambda}^{\C} (z) = 0\}. \end{equation}
We also denote the complexification of a real analytic submanifold $H$   by $H_{\C}.$
We  further denote the restriction of an eigenfunction to $H$ by $\phi_j |_{H} $ or equivalently by $\gamma_H^* \phi_j$ and the holomorphic extensions by $\phi_j^{\C} |_{H_{\C}}$.

Let $\alpha_H: H \to M$ be a real analytic paramaterization of 
a real analytic submanifold. In the case of a curve $\ccal$, we use  the  complexification of an arc-length parameterization,
\begin{equation} \label{GAMMAX} \alpha_{\ccal}: \R \to M  \end{equation} 
 The parametrization extends to some strip  $ S_{\epsilon} = \{(t
+ i \tau \in \C: |\tau| \leq \epsilon\} $   as a holomorphic curve
\begin{equation} \label{alphaXCX} \alpha_{H}^{\C}: S_{\epsilon} \to M_{\epsilon}.  \end{equation}
We let $\tau_H $ be the maximal $\epsilon$ for which there exists an 
analytic extension of $\alpha_C$.

The intersection points of $\alpha_{H_{\C}}$ and
$\ncal_{\phi_j}^{\C}$ correspond to the zeros of the pullback
$(\alpha_{H}^{\C})^* \phi_j^{\C}$. When  $\ccal$ is a good curve,
then $\phi_{j}^{\C} |_{\ccal_{\C}}$ has
a discrete set of zeros, we can define the current of summation
over the zero set by
\begin{equation} \label{NCALCURRENT} [\ncal^{\ccal}_{\lambda_j}] = \sum_{ \{ t + i \tau :\; \phi_j^{\C}(\alpha_{\ccal^{\C}}(t + i \tau)) = 0 \} } \delta_{t + i \tau}.
\end{equation}

 Slightly modifiying the definition \eqref{ujdefCC}, we define  the sequence 
\begin{equation} \label{vxxi} v_j^{\ccal} : = \frac{1}{\lambda_{j} } \log \left| \alpha_{\ccal}^* \phi_{\lambda_j}^{\C} (t + i \tau)
\right|^2\end{equation}
of subharmonic functions on the strip $S_{\epsilon} \subset \C$.

By the  Poincar\'e-Lelong formula, 
\begin{equation}\label{PLL}  [\ncal^{\ccal}_{\lambda_j}] = \frac{ i}{\pi} \ddbar_{t + i
\tau} \log \left| \alpha_{\ccal}^* \phi_{\lambda_j}^{\C} (t + i \tau)
\right|^2. \end{equation} 

Put:
\begin{equation} \label{acal} \acal_{L, \epsilon}  (\frac{1}{\lambda} dd^c \log |\phi_j^{\C}|^2) = \frac{1}{\lambda} \int_{S_{\epsilon, L}}
dd^c_{t + i \tau} \log |\phi_j^{\C}|^2 (\alpha_{\ccal} (t + i \tau) ). \end{equation}  
 To prove that $ n(\phi_{\lambda}^{\C}, \ccal_{\C}) \leq A \lambda$, it suffices to show that  there exists $M < \infty$ so that
\begin{equation} \label{acalest} \acal_{L, \epsilon}  (\frac{1}{\lambda} dd^c \log
|\phi_j^{\C}|^2)  \leq M. \end{equation}

To prove \eqref{acalest}, we observe that since  $dd^c_{t + i \tau} \log |\phi_j^{\C}|^2 (\alpha_{\ccal}  (t + i \tau))$ is a positive $(1,1)$ form on the strip, the integral over
$S_{\epsilon}$ is only increased if we integrate against a
positive smooth test function $\chi_{\epsilon} \in
C_c^{\infty}(\C)$ which equals one on $S_{\epsilon, L}$ and vanishes
off $S_{2 \epsilon, L} $. Integrating  by
parts the $dd^c$ onto $\chi_{\epsilon}$, we have 
$$\begin{array}{lll} \acal_{L, \epsilon} (\frac{1}{\lambda} dd^c \log |\phi_j^{\C}|^2) &\leq &  \frac{1}{\lambda}  \int_{\C}
dd^c_{t + i \tau} \log |\phi_j^{\C}|^2 (\alpha_{\ccal} (t + i \tau) )
\chi_{\epsilon} (t + i \tau) \\ && \\  &= & \frac{1}{\lambda}  \int_{\C}
 \log |\phi_j^{\C}|^2 (\alpha_{\ccal} (t + i \tau) )
dd^c_{t + i \tau} \chi_{\epsilon} (t + i \tau) .
\end{array}$$

To complete the proof of \eqref{acalest} it suffices  to prove that
 \begin{equation} \label{logbound}
\limsup_{\lambda_j \to \infty} \frac{1}{\lambda_j} \, \Big| \log |\phi_j^{\C}|^2 (\zeta) \Big| \leq C, \quad \zeta \in S_{\epsilon} \end{equation}
for some $ C> 0$.
Now write $\log |x| = \log_+ |x| - \log_- |x|$. Here $\log_+ |x| = \max\{0, \log |x|\}$ and
$\log_- |x|= \max\{0, - \log |x| \}. $ In view of (\ref{logbound}),   we need upper bounds for 
$$  \frac{1}{\lambda} \int_{\C}
 \log_{\pm} |\phi_j^{\C}|^2 (\alpha_{\ccal}(t + i \tau) )
dd^c_{t + i \tau} \chi_{\epsilon} (t + i \tau).$$

For $\log_+$ the upper bound is an immediate consequence the global upper bound
\begin{equation} \label{UBa} \limsup_{k \to \infty} \frac{1}{\lambda_j} \log |\phi^{\C}_{j_k}(\zeta)|^2 \leq  2 \sqrt{\rho}(\zeta)\end{equation}
 proved in \cite{ZPl} using the complexified wave (ie. Poisson operator). 
 Here, $\sqrt{\rho}$ is the Grauert tube function of $(M,g)$.
On the complexified curve or strip, one lets $A = \sup_{\ccal_{\tau}} \sqrt{\rho} < \infty$ where $\ccal_{\tau}$ is the intrinsic Grauert tube of raidus $\tau$ of the curve,
which in general is not defined by the same as the Grauert tube radius $\sqrt{\rho}$ of
$(M, g)$. The proof is valid for any $\tau > 0$ less than the maximal radius
of analytic continuation of the curve.

For $\log_-$ the lower bound follows from the $\scal$-good assumption
that $\log_- |\phi_j^{\C} |\leq A \lambda_j$. This establishes the bound in (\ref{logbound}) and completes the proof of Theorem \ref{INTER}.

\medskip

\section{\label{HYPER} Proof of Theorem \ref{NODALBOUND}}

In higher dimensions, we use Crofton's formula to prove that  for
$\scal$-good hypersurfaces, 
$\hcal^{m-2}(\NPHI \cap H) $ is bounded above by a certain measure
of the complexified nodal set. We closely follow \cite{ZSt} and refer
there for some of the background.  The principal difference is that we let $(H, g_H)$ with $g_H: = g|_{TH}$
be the Riemannian manifold of \cite{ZSt} instead of $(M,g)$. { Thus, $H \cap \NPHI$ 
is   a real analytic hypersurface (i.e. a real analytic variety of codimension one)  of $H$ in the sense that it is
the subset $\{\phi_j |_H = 0\} \subset H$ defined by the analytic function
$\phi_j |_H$.  The $\scal$-goodness assumption on $H$ implies that
$H \cap \NPHI$ has codimension one since it certainly implies that the restricted analytic function is non-zero.
 It may have a singular set of codimention one in $H \cap \NPHI$.\footnote{One way to prove this is to use Whitney's stratification
 theorem \cite{K}. For hypersurfaces there is probably a simpler proof.}
In the following, we  write $N = H \cap \NPHI$. We retain
$m = \dim M$ so that $m-1 = \dim H$. 

\subsection{Crofton formula} The main result of this section is
Proposition \ref{CROFTONEST}.  To prepare for the statement and proof,
 we introduce some notation and make
some useful observations. Most are from \cite{ZSt} in a section on general
hypersurfaces in Riemannian manifolds, and hence they apply to $N \subset H$ with only a change of notation. We recall some of the statements for
the sake of completeness.

Let $\pi: T^* H \to H$ be the natural projection
 We  denote by $\omega$ the standard symplectic form on
$T^* H$ and by $\alpha$ the canonical one form. As above, we denote by $d\mu_L$  the Liouville measure on $S^*
H$. Then 
$d\mu_L = \omega^{n-1} \wedge \alpha$ on $S^* H$.
We also denote the Hamiltonian generating the geodesic flow $G_H^t$ by the Hamiltonian $ |\xi|_{g_H}$ and its  Hamilton vector field by $\Xi = \Xi_H$.
Note that it is quite different from the geodesic flow of $(M,g)$.

Let $N \subset H$ be a smooth hypersurface in a Riemannian manifold $(H,
g_H)$.  We denote by  $T^*_N H$ the 
of covectors with footpoint on $N$ and $S^*_N H$ the unit covectors along $N$.
We   introduce  Fermi normal coordinates $(s, x_m) $
along
 $N \subset H$,  where  $s$ are coordinates on $H$ and $x_{m-1}$  is the
 normal coordinate, so that $x_{m-1} = 0$ is a local defining function for $N$.   We also let $\sigma, \xi_{m-1}$ be
 the dual symplectic Darboux coordinates. Thus the canonical
 symplectic form is $\omega_{T^* H } = ds \wedge d \sigma + dx_{m-1}
 \wedge d \xi_{m-1}. $
.

\begin{lem} \label{NSYMP} The restriction $\omega |_{S_N^* H}$ is symplectic on $S^*_N H \backslash S^* N$. 
\end{lem}
Indeed, $\omega |_{S_N^* H}$ is symplectic on $T_{y, \eta} S^* H$ as long as $T_{y, \eta} S^*_N H$ 
is transverse to $\Xi_{y, \eta}$, since $\ker (\omega|_{S^*M}) = \R \Xi. $

It follows from Lemma \ref{NSYMP} that the symplectic volume form of $S^*_H M \backslash S^* N$
is $\omega^{m-2} |_{S_N^* M}$. The following Lemma gives a useful alternative formula:

\begin{lem} \label{dmuLN} 
Define  $$d\mu_{L, N} =  \iota_{\Xi} d\mu_L \;|_{S^*_N H}, $$
where  as above,   $d\mu_L$ is Liouville measure on  $S^* H$.  Then $$d \mu_{L, N}= \omega^{m-2} |_{S_N^* H}. $$
\end{lem}

Indeed,  $d \mu_L = \omega^{m-2} \wedge \alpha$, and $ \iota_{\Xi} d\mu_L = \omega^{m-2}$.  As in \cite[Corollary 8]{ZSt}, 
\begin{equation} \label{HNN}\begin{array}{lll} \hcal^{m-2}(N ) & = & \frac{1}{\beta_m}
\int_{S^*_N H} |\omega^{m-2}|. \end{array} \end{equation}

    As reviewed in \cite{ZSt}, a Crofton
formula arises from a double fibration

$$\begin{array}{lllll} && \ical  && \\ &&&&\\
& \pi_1 \;\swarrow & & \searrow \;\pi_2 & \\ &&&& \\ \Gamma  &&&&
B,
\end{array}$$
where $\Gamma$ parametrizes a family of submanifolds $B_{\gamma}$ of $B$. The points  $b \in B$
then parametrize a family of submanifolds $\Gamma_b = \{\gamma \in \Gamma: b \in B_{\gamma}\}$
and the top space is the incidence relation in $B \times \Gamma$ that $b \in B_{\gamma}.$ See  \cite{AB,AP} for background.

We would like to define $\Gamma $ as the space of geodesics of  $H$. This is not a Hausdorff space, so instead of we defined $\Gamma$ to be the
set of $H$-geodesic arcs of some fixed length $L $ (less than the injectivity radius $L_1$ of $H$).

The relevant Crofton formula is the following,

\begin{prop}\label{CROFTONEST}  Let $ N \subset H$ be a real analytic irreducible  hypersurface \footnote{The same formula is true
if $ N$ has a singular set $\Sigma$ with $\hcal^{m-2}(\Sigma)  = 0$}, and let $S^*_N H$ denote
the unit covers to $M$ with footpoint on $N$. Then for
$0 < T < L_1,$ 
$$\hcal^{m-2}( N)  = \frac{ 1}{\beta_mT}  
 \int_{S^* H}
\# \{t \in [- T, T]: G_H^t(x, \omega) \in S^*_N H\} \, d\mu_{L}(x,
\omega),$$
where $\beta_m $ is $2 (m-1)!$ times  the volume of the unit ball in $\R^{m-2}$. 
\end{prop}

\begin{proof}

We argue as in \cite[Proposition 9]{ZSt} and repeat some of the
arguments there to keep the proof self-contained.  We  define the
incidence relation $$ \ical_T = \{((y, \eta),  (x, \omega), t)
\subset S^*H \times S^*H \times [- T, T]:  (y, \eta) =   G_H^t(x,
\omega)\}, $$ and then define $\ical_{T, N}$ by restricting
 $x \in  N$.
We then consider the diagram,  
\begin{equation} \label{DIAGRAM} \begin{array}{lllll} && \ical_{T} \simeq S^* H \times [-T, T]  && \\ &&&&\\
& \pi_1 \;\swarrow & & \searrow \;\pi_2 & \\ &&&& \\ (S^* H) &&&&
S^* H,
\end{array} \end{equation}
where 
 $$\pi_1(t, x, \xi) = G_H^t(x, \xi), \;\;\; \pi_2(t, x, \xi) = (x, \xi), $$ 

and restrict it to $S^*_N H$ to obtain
\begin{equation} \label{DIAGRAM2} \begin{array}{lllll} && \ical_{T,N} \simeq S^* H \times [-T, T]  && \\ &&&&\\
& \pi_1 \;\swarrow & & \searrow \;\pi_2 & \\ &&&& \\ (S_N^* H)_T &&&&
S_N^* H,
\end{array} \end{equation}
 where
$$(S^*_N H)_{T} = \pi_1 \pi_2^{-1} (S_N^* H) =  \bigcup_{|t| < T} G_H^t(S^*_N H).$$

 We  define the Crofton density $\phi_T$ on $S_N^* H$ corresponding to   the diagram \eqref{DIAGRAM} \cite{AP} (section 4) by 
\begin{equation} \label{CROFDEN} \phi_T = (\pi_2)_* \pi_1^*
d\mu_L. \end{equation}  $\phi_T$
is a differential form of dimension $2 \dim H - 2$ on $S^*H$.  Let  $\chi$  be a smooth cutoff  equal to $1$ on $(- \half, \half)$, and let
$\chi_T(t) = \chi(\frac{t}{T}). $  Then a smooth version of \eqref{CROFDEN} is $\pi_1^* (d\mu_L \otimes
\chi_T dt)$ is a smooth density on $\ical_{T,N}$. As in \cite[Lemma 10]{ZSt} one has,

\begin{lem} \label{phiT} The Crofton density \eqref{CROFDEN} is given by, $\phi_T = T d\mu_{L, N} $ \end{lem}

 Combining Lemma \ref{phiT} with \eqref{HNN} gives
\begin{equation}  \label{HDPHIT} \int_{S^*_N H} \phi_T = \int_{\pi_2^{-1} (S^*_N H)} d\mu_L =  T\beta_m\hcal^{m-2}(N).  \end{equation}

We then relate the integral on the left side to numbers of intersections of $H$-geodesic arcs with $N$. 
By the co-area formula (see \cite[Section 3.2]{ZSt},
\begin{equation} \label{COAREA} \int_{\pi_2^{-1}(S^*_N H)} \pi_1^* d\mu_L = \int_{S^* H}
\# \{t \in [- T, T]: G^t(x, \omega) \in S^*_N H\} d\mu_L(x,
\omega). \end{equation} 

Combining \eqref{HDPHIT} and \eqref{COAREA} gives the  result stated in Proposition \ref{CROFTONEST}.


\end{proof}

\subsection{Complexification}

The next step is to complexify geodesics of $H$ and also the nodal set
$N = \NPHI$. Here, geodesics and exponential maps always refer to geodesics of $H$.

Define
$$F: S_{\epsilon} \times S^*H \to H_{\C}, \;\;\; F(t + i
\tau, x, v) = \exp_x (t + i \tau) v, \;\;\; (|\tau| \leq \epsilon)
$$ 

Let $\sqrt{\rho}_H$ be the Grauert tube funciton of $H$, which in general is distinct from the ambient Grauert tube function of $(M,g)$ denoted above by $\sqrt{\rho}.$  Let $H_{\tau} = \{z \in H_{\C}: \sqrt{\rho}_H(z) < \tau\}$ be the intrinsic Grauert tube of radius $\tau$ of $H$.

 For each $(x, v) \in S^* H$,
$$F_{x, v}(t + i \tau) = \exp_x (t + i \tau) v $$
is a holomorphic strip contained in $H_{\tau}$. Here, $S_{\epsilon} = \{t + i \tau \in \C:
|\tau| \leq \epsilon\}. $ We also denote by
 $S_{\epsilon, L} = \{t + i \tau \in \C:
|\tau| \leq \epsilon, |t| \leq L \}. $

  Since $F_{x, v}$ is a holomorphic function in the strip $S_{\epsilon},$  by Poincar\'{e}-Lelong, 
$$F_{x, v}^*(\frac{1}{\lambda} dd^c \log |\psi_j^{\C}|^2) =
 \frac{1}{\lambda} dd^c_{t + i \tau} \log |\phi_j^{\C}|^2 (\exp_x (t + i \tau)
 v) = \frac{1}{\lambda} \sum_{t + i \tau: \phi_j^{\C}(\exp_x (t + i
 \tau) v) = 0} \delta_{t + i \tau}.
$$ As in \eqref{acal}, put
\begin{equation} \label{acal2} \acal_{L, \epsilon}  (\frac{1}{\lambda} dd^c \log |\phi_j^{\C}|^2) = \frac{1}{\lambda} \int_{S^* H} \int_{S_{\epsilon, L}}
dd^c_{t + i \tau} \log |\phi_j^{\C}|^2 (\exp_x (t + i \tau) v)
d\mu_L(x, v). \end{equation}

A key observation of \cite{DF, Lin} is that
(with $\ncal_{\lambda}: = \NPHI$)  for any $(x,v) \in S^*H,$ 
\begin{equation} \label{MORE2}
\#\{\ncal_{\lambda}^{\C} \cap F_{x,v}(S_{\epsilon, L}) \} \geq \#\{\ncal_{\lambda}^{\R} \cap F_{x,v}(S_{0, L}) \}, 
\end{equation}
since every real zero is a complex zero. 
It follows then from    Proposition \ref{CROFTONEST} (with $N = \ncal_{\lambda}$) that

$$\begin{array}{lll} \acal_{L, \epsilon}  (\frac{1}{\lambda} dd^c \log
|\phi_j^{\C}|^2) &= &  \frac{1}{\lambda} \int_{S^* H}
\#\{ t + i\tau \in S_{\epsilon,L}; \,  F_{x,v}(t+i\tau) \in \ncal_{\lambda}^{\C} \} \,  d
\mu_L(x,v) \\ && \\
&\geq &  \frac{1}{\lambda} \hcal^{m-2}(\ncal_{\lambda} \cap H).\end{array} $$

Hence to obtain an upper bound on $\frac{1}{\lambda} \hcal^{m-2}(\ncal_{\lambda} \cap H),$ it suffices
to prove that there exists $M < \infty$ so that
\begin{equation} \label{acalest2} \acal_{L, \epsilon}  (\frac{1}{\lambda} dd^c \log
|\phi_j^{\C}|^2)  \leq M. \end{equation}

To prove \eqref{acalest2}, we observe that since  $dd^c_{t + i \tau} \log |\psi_j^{\C}|^2 (\exp_x (t + i \tau)
v)$ is a positive $(1,1)$ form on the strip, the integral over
$S_{\epsilon}$ is only increased if we integrate against a
positive smooth test function $\chi_{\epsilon} \in
C_c^{\infty}(\C)$ which equals one on $S_{\epsilon, L}$ and vanishes
off $S_{2 \epsilon, L} $. Integrating  by
parts the $dd^c$ onto $\chi_{\epsilon}$, we have 
$$\begin{array}{lll} \acal_{L, \epsilon} (\frac{1}{\lambda} dd^c \log |\phi_j^{\C}|^2) &\leq &  \frac{1}{\lambda} \int_{S^* H} \int_{\C}
dd^c_{t + i \tau} \log |\phi_j^{\C}|^2 (\exp_x (t + i \tau) v)
\chi_{\epsilon} (t + i \tau) d\mu_L(x, v) \\ && \\  &= & \frac{1}{\lambda} \int_{S^* H} \int_{\C}
 \log |\phi_j^{\C}|^2 (\exp_x (t + i \tau) v)
dd^c_{t + i \tau} \chi_{\epsilon} (t + i \tau) d\mu_L(x, v) .
\end{array}$$

As in the case of curves, we need upper bounds for 
$$  \frac{1}{\lambda} \int_{S^* H} \int_{\C}
 \log_{\pm} |\phi_j^{\C}|^2 (\exp_x (t + i \tau) v)
dd^c_{t + i \tau} \chi_{\epsilon} (t + i \tau) d\mu_L(x, v) .$$
For $\log_+$ the upper bound is an immediate consequence of \eqref{UBa}. 

  For $\log_-$, we use the assumption that $H$ is a good hypersurface, which implies that for any smooth function $J$  there exists $C > 0$ so that  \begin{equation}
\label{LOGINT} \frac{1}{\lambda} \int_{H_{\tau} } \log |\phi^{\C}_{\lambda}| J d V \geq - C.  \end{equation}
We then rewrite \eqref{acal2} to show that \eqref{LOGINT} gives the same
lower bound $-C$ for \eqref{acal2}.

We use the  diffeomorphism  $E: B_{\epsilon}^*H \to H_{\epsilon}$ defined by
$E(x, \xi) = \exp_x i \xi$. Since $B_{\epsilon}^* H = \bigcup_{0 \leq \tau \leq \epsilon} S^*_{\tau} H$
we also have that 
$$E:   S_{\epsilon, L} \times S^* H  \to H_{\tau}, \;\;\; E(t + i \tau, x, v) = \exp_x (t + i \tau) v  $$
is  a diffeomorphism for each fixed $t$. Hence  by letting $t$ vary, $E$ 
is a smooth fibration with  fibers given by  geodesic arcs.  Over a point $\zeta \in H_{\tau}$ the fiber of the map is a geodesic arc
$$\{ (t + i \tau, x, v): \exp_x (t + i \tau) v = \zeta, \;\; \tau = \sqrt{\rho}_H(\zeta)\}. $$ Pushing forward the measure $
dd^c_{t + i \tau} \chi_{\epsilon} (t + i \tau) d\mu_L(x, v) $ under $E$ gives  a  measure $d\omega$ on $H_{\tau}$, and as in \cite{ZSt},
\begin{equation}\label{PUSH}  \omega: = E_* \; dd^c_{t + i \tau} \chi_{\epsilon} (t + i \tau) d\mu_L(x, v)  =\left (\int_{\gamma_{x, v}} 
\Delta_{t + i \tau} \chi_{\epsilon} ds \right) dV, \end{equation}
where $dV$ is the  K\"ahler volume form on $H_{\epsilon}$.
In particular  it is a smooth  multiple $J$  of the K\"ahler volume form $dV$.
It follows that
\begin{equation} \label{JEN}  \int_{S^* H} \int_{\C}
 \log |\phi_j^{\C}|^2 (\exp_x (t + i \tau) v)
dd^c_{t + i \tau} \chi_{\epsilon} (t + i \tau) d\mu_L(x, v) = \int_{H_{\tau}} 
\log |\phi_j^{\C}|^2   J d V.  \end{equation}

It follows that \eqref{acal} is bounded above and below, completing the proof of Theorem \ref{NODALBOUND}.

\section{\label{BACKGROUND} Background on asymmetry and the geometry of flowouts}

For the remainder of the article we prove Theorems \ref{MASSMICRO}-\ref{mainthm1}. 
In this section we review the  the geodesic  asymmetry condition
of Definition \ref{MADEF}. We further consider the geometry of 
the condition \eqref{ASSUME}.
We begin with some background from \cite{TZ2}.

Let $(s, y_n)$ denote Fermi normal coordinates on $H = \{y_n = 0\}$
and let $\sigma, \eta_n$ denote the dual symplectic coordinates. Define
\begin{equation}\label{gammaDEF} \gamma(s, y_n, \sigma, \eta_n) =  \frac{|\eta_n|}{\sqrt{|\sigma|^2
+ |\eta_n|^2}} = (1 - \frac{|\sigma|^2}{r^2})^{\half} ,\;\;\;(r^2 = |\sigma|^2 + |\eta_n|^2)  \end{equation}    
on $T^*_H M$ and also denote by
\begin{equation} \label{gammaBH} \gamma_{B^*H} = (1 - |\sigma|^2)^{\half} \end{equation}
its restriction to $S^*_H M = \{r = 1\}$. 

We denote by $G^t$ the homogeneous  geodesic flow of $(M, g)$,
i.e.  Hamiltonian flow on $T^*M - 0$ generated by $|\xi|_g$. We
then put $\exp_x t \xi = \pi \circ G^t(x, \xi)$.  We further denote by
\begin{equation} T^*_H M = \{(q, \xi) \in T_q^* M, \;\; q\in H\} \end{equation}
 the covectors to $M$ with footpoint on $H$, and by
$T^* H=  \{(q, \eta) \in T_q^* H, \;\; q\in H\}$ the cotangent
bundle of $H$.  We further denote by $ \pi_H : T^*_H M \to T^*
H $  the restriction map,
\begin{equation} \label{RESCV}  \pi_H(x, \xi)  = \xi |_{TH}.
\end{equation}
It is a linear map whose kernel is the conormal bundle $N^* H$ to
$H$, i.e. the annihilator of the tangent bundle $TH$. In the
presence of the metric $g$,  we may identify co-vectors in $T^*M$
with vectors in $TM$ and  induce a co-metric $g$ on $T^*M$. The
orthogonal decomposition  $T_H M = T H \oplus N H$ induces an
orthogonal decomposition $T_H^* M = T^* H \oplus N^* H, $ and the
restriction map (\ref{RESCV}) is  equivalent modulo metric
identifications to the tangential orthogonal projection (or restriction)
\begin{equation} \label{piHdef} \pi_H: T^*_H M \to T^* H. \end{equation}

For any orientable (embedded) hypersurface $H \subset M$, there
exists two unit normal co-vector fields $\nu_{\pm}$ to $H$ which
span half ray bundles $N_{\pm} = \R_+ \nu_{\pm} \subset N^* H$.
Infinitesimally, they define two `sides' of $H$, indeed they are
the two components of  $T^*_H M \backslash T^* H$. We often use
Fermi normal coordinates $(s, y_n)$ along $H$ with $s \in H$ and
with $x = \exp_x y_n \nu$. We let $\sigma, \eta_n$ denote the dual
symplectic coordinates.

We also denote by $S^*_H M,$ resp. $S^* H$, the unit covectors in
$T^*_H M$, resp. $T^* H$. We restrict \eqref{piHdef} to get  $\pi_H : S^*_H M \to B^* H,$ with
 where $B^* H$ is the unit coball bundle of $H$.
Conversely, if $(s, \sigma) \in B^* H$, then there exist two unit
covectors $\xi_{\pm}(s, \sigma) \in S^*_s M$ such that
$|\xi_{\pm}(s, \sigma)| = 1$ and $\xi|_{T_s H} = \sigma$. In the
above orthogonal decomposition, they are given by

\begin{equation} \label{xipm} \xi_{\pm}(s, \sigma) = \sigma \pm \gamma(s,\sigma)
\nu_+(s),  \,\,\,\, \gamma(s,\sigma):= \sqrt{1 - |\sigma|_s^2}. \end{equation} \

We define the reflection involution
through $T^* H$ by
\begin{equation}\label{rHDEF}  r_H: T_H^* M \to T_H^* M, \;\;\;\;
r_H(s, \mu\; \xi_{\pm}(s, \sigma)) =(s, \mu\; \xi_{\mp}(s,
\sigma)), \,\,\,  \mu \in \R_{+}.
\end{equation} Its  fixed point
set is $T^* H$.

 We define the {\it first return time}
$T(s, \xi)$ on $S^*_H M$ by,
\begin{equation} \label{FRTIME} T(s, \xi) = \inf\{t > 0: G^t (s,
\xi) \in S^*_H M, \;\;\ (s, \xi) \in S^*_H M)\}. \end{equation} By
definition $T(s, \xi) = + \infty$ if the trajectory through $(s,
\xi)$ fails to return to $H$.  We define the first return map  by
\begin{equation} \label{FIRSTRETURN} \Phi: S^*_H M \to S^*_H M, \;\;\;\; \Phi(s, \xi) = G^{T(s, \xi)} (s, \xi) \end{equation} Inductively, we
define the jth return time $T_j(s, \xi)$ to $S^*_H M$ and the
jth return map $\Phi^j$ when the return times are finite. 

We further define the `first impact time'  on all of $S^* M$,
\begin{equation} \label{FIRSTIMPACT} t_1(x, \xi) =  \left\{ \begin{array}{l}  \inf \{t \geq 0, G^{t}(x, \xi) \in S^*_H M\}, \\ \\ = +\infty, \; \rm{if\;no\; such\; t\; exists} \end{array} \right. \end{equation}
Note that $t_1$ is lower semi-continuous, so that its sublevel
sets $\{t_1 \leq  \alpha\}$ are closed.
Similarly, define 
  $t_j(x, \xi)$ to be the jth
 `impact time', i.e. the  time to the $j$th impact with $H$. By
homogeneity of  $G^{t}: T^*M \rightarrow T^*M$,  for all $j \in
\Z,$
\begin{equation} \label{impact}
t_{j}(x, \xi) = t_j(x, \frac{\xi}{|\xi|}); \, \, \xi \neq 0.
\end{equation}
Obviously, $t_j(x, \xi)  = t_1(x, \xi) + T_j(G^{t_1(x, \xi)}(x, \xi)). $

Define
 
 \begin{equation} \label{relationi} \left\{\begin{array}{l} \Delta _{T^*M \times T^* M} := \{ (x,\xi,x,\xi) \in T^*M \times T^*M \}, \\ \\
 \Gamma_T =
\bigcup_{(s, \xi) \in T^*_H M}
 \bigcup_{ |t|  < T} \{( G^t(s, \xi),  G^t(r_H (s, \xi)) \}.
\end{array} \right. \end{equation} 
The two `branches' or components intersect along the singular set
 \begin{equation} \label{SIGMA} \Sigma_T : = \bigcup_{|t| < T}  (G^t \times G^t) \Delta_{T^* H \times T^* H}. \end{equation}
We further subscript  $\Gamma_T$ with $\epsilon$ to indicate the points $\Gamma_{T, \epsilon}$  making an angle $\geq \epsilon$ with $T H$.

Since $ G^t(r_H (s, \xi)) = G^t r_H G^{-t} G^{t}
 (s,\xi)$,
 $\Gamma_{T, \epsilon} \subset  \Gamma_{T} \backslash \Sigma_T$ is the graph of a
 symplectic correspondence.  More precisely, for any $\epsilon > 0$,  $\Gamma_{T, \epsilon}$ is
 the union of a finite number $N_{T, \epsilon}$  of graphs of partially defined
 canonical transformations
 \begin{equation} \label{RCAL} \rcal_j(x,\xi) = G^{t_j(x,\xi)} r_H G^{-t_j(x,\xi)}(x,\xi).\end{equation}
which we term $H$-reflection maps.

\subsection{\label{ASYMSECT} Asymmetric hypersurfaces}

 In the following, $\mu_L$  denotes Liouville measure on $S^*M$ and $\mu_{L,H}$ is the induced hypersurface measure on $H$ satisfying $ d\mu_L =  d\mu_{L,H} \, dx_n.$

\begin{defin} \label{ANC}  We say that  $H$ has a positive
measure of microlocal reflection symmetry if 
$$  \mu_{L, H} \left( \bigcup_{j \not= 0}^{\infty}  \{(s, \xi) \in S^*_H M : r_H G^{T_j(s, \xi)} (s, \xi)  =
 G^{T_j (s, \xi)} r_H (s, \xi)  \}\right) > 0.  $$ 
Otherwise we say that $H$ is asymmetric with respect to the geodesic flow. 

\end{defin} Thus, the return time condition is that the $+$ and $-$ trajectories return at the same time to the same point of $H$ and project to
the same covector in $B^* H$ on a set of positive measure.

 \subsection{Filtering the flowout by return times and by tangential angle}

We recall that our full-measure flowout assumption  \eqref{ASSUME} is 
 $\mu_L ( \rm{FL}(H)) = 1.$
Since $ \cup_{|t| < \infty} G^t(S^*H)$ has Hausdorff dimension $\leq 2n-2,$ it follows that
$$ \mu_L \big( \bigcup_{|t| < \infty} G^t(S^*H) \big) = 0,$$
and so, in particular,
$$\mu_L ( \{ (x,\xi) \in { \rm{FL} }(H),  G^{t_1(x,\xi)}(x,\xi) \in S^*H \} ) = 0.$$\\
Here,   $t_1(x,\xi)$ is the first hitting time \eqref{FIRSTIMPACT}.
Let
\begin{equation} \label{LAMBDADEF}\left\{ \begin{array}{l}
\overline{\Lambda}= \{ (x,\xi) \in S^*M, \,\, |t_1(x,\xi)| < \infty,  \}  \\ \\
   \Lambda: = \{ (x,\xi) \in S^*M, \,\, |t_1(x,\xi)| < \infty, \,\, G^{t_1(x,\xi)}(x,\xi) \in S^*_H M \setminus S^*H \} \end{array} \right.. \end{equation}  Here,
   $\overline{\Lambda}$ is the set of covectors whose orbits hit $H$ at some time,   and $\Lambda \subset \overline{ \Lambda}$ is the subset which  never
tangentially. Evidently, $\Lambda \subset \rm{FL}(H) \subset \overline{\Lambda}$ and the differences of these sets have measure zero.
Then \eqref{ASSUME} is equivalent to 

\begin{equation} \label{ASSUMEb}
\mu_L ( \Lambda ) = 1. \end{equation}\

 One can clearly make the decomposition
$$ \Lambda = \bigcup_{R=0}^{\infty} \Lambda_R, \quad \Lambda_R:=   \{ (x,\xi) \in \Lambda, \,\, | t_1(x,\xi) | < R \}.$$
Moreover, for all $R_1 \leq R_2 \leq R_3 \leq \dots,$ the sets $\Lambda_{R_1} \subset \Lambda_{R_2} \subset \Lambda_{R_3} \subset \dots$ and so, by monotonicity of measure,
\begin{equation} \label{monotone1}
\mu_L (\Lambda_R) \nearrow 1 \quad \text{as} \, \;R\to \infty.\end{equation}
 One can make a further decomposition
 $$ \Lambda_R = \bigcup_{\ep} \Lambda_{R,\epsilon}, \quad \Lambda_{R,\ep}:= \{(x,\xi) \in \Lambda_R, \, G^{t_1(x,\xi)}(x,\xi) \in S^*_H M, \,  \,  | \pi_H ( G^{t_1(x,\xi)}(x,\xi) )| < 1 - 2\epsilon \}.$$
 Since $\Lambda_{R,\epsilon_1} \subset \Lambda_{R,\epsilon_2} \subset \cdots$ and $ \epsilon_1 \geq \epsilon_{2} \geq \cdots$ it follows again by monotonicity that
 \begin{equation} \label{monotone2}
  \mu_L (\Lambda_{R,\epsilon}) \nearrow \mu_{L}(\Lambda_R) \quad \text{as} \, \ep \to 0^+.\end{equation}
  
  Thus from (\ref{monotone1}) and (\ref{monotone2}) it follows that for any $\delta \in (0, 1/2)$ one can choose $R= R(\delta)$ and $\epsilon = \epsilon(\delta)$ such that
  \begin{equation} \label{measure1}
 \mu_{L}( \Lambda_{R,\epsilon}) \geq 1- 2 \delta. \end{equation}

 We will need the following facts about   $\Lambda \subset S^*M$ in (\ref{LAMBDADEF}):
 \begin{lem} \label{open} We have:
 
 \begin{enumerate}

  \item $\Lambda$ is open.
 
\item   The first impact time $t_{1}|_{\Lambda}$ is  $C^{\infty}$ on $\Lambda$.

 \end{enumerate}
 \end{lem}
 
\begin{rem} Open-ness is not obvious, since $t_1$ is lower semi-continuous
and has open super-level sets $\{t_1 > \alpha\}$. This is not a contradiction,
since the tangential directions are punctured out in $\Lambda$ and they form its boundary.
\end{rem}
 
 \begin{proof}

Let  $\rho \in C^{\infty}(M)$ be a defining function for $H$, i.e.
 $$H= \{ x \in M; \rho(x) = 0 \}, \quad d\rho(x) \neq 0,  \,\, x \in H.$$
 
 Let $(x_0,\xi_0) \in \Lambda \subset S^*M$ and so, in particular, $
 t_1(x_0,\xi_0)| < \infty.$ We claim that there exists an open set $U$
 around $(x_0,\xi_0)$ so that $U \subset \Lambda$.
 
 Consider the map $G: \R \times S^*M \to S^*M$ given by $G(t, (x,\xi)) = G^t(x,\xi)$ and let $\pi: S^*M \to M$ be the canonical projection. Let
 $\gamma_{x, \xi} (t) = \pi G^t(x, \xi)$ and consider the sets,
 
 \begin{equation} \label{ift} \left\{ \begin{array}{l} C =\{(t, x,\xi) \in \R \times S^*M : 
 \rho( \pi G(t, (x,\xi)) ) = 0\}, \\ \\ 
C_0=\{(t, x,\xi) \in C:
d \rho(\dot{\gamma}_{x, \xi}(t)) \not=0\} \end{array} \right. \end{equation} Then $\overline{\Lambda}  = 
 G(C) $ and $\Lambda \subset  G(C_0)$.  Note that $C$ is closed but $G(C)$ is generally not  a closed subset of $S^* M$ since $G$ is
 not proper. 
 
 The t-derivative of $ \rho( \pi G(t, (x,\xi)) )  = \pi^* \rho G^t(x, \xi)$  is given by,
 \begin{equation} \label{drho}  \partial_t  \rho( \pi G(t, (x,\xi)) ) = H_{|\xi|_g} \pi^*\rho (x, \xi ) = d \rho_x
( \dot{\gamma}_{x, \xi}(t)) , \end{equation}
 where $H_{|\xi|_g}$ is the Hamilton vector field and $\pi_* H_{|\xi|_g} (G^t(x, \xi)= \dot{\gamma}_{x, \xi}(t)$. Hence, \eqref{drho} is non-zero for
 $(t_0, x_0, \xi_0) \in C_0$. By the   implicit function theorem, there exists an
 open set $U_{x_0, \xi_0} \subset S^* M$ around $(x_0, \xi_0)$ on which 
 there exists  a $C^\infty$ function $\tilde{t}: U_{x_0, \xi_0} \to \R$ satisfying $\tilde{t}(x_0, \xi_0) = t_0$ and $\rho(\pi G^{t(x, \xi)}(x,  \xi)) = 0$.
 
 Now suppose that $(x_0, \xi_0) \in \Lambda$. Then $(t_1(x_0, \xi_0), x_0, \xi_0) \in C_0$ and $\tilde{t} = t_1$ on $U_{x_0, \xi_0}$. Then $t_1$ is 
 $C^{\infty}$ on $U_{x_0, \xi_0}$ and in particular is finite. Hence, $U_{x_0, \xi_0} \subset \Lambda. $

 \end{proof}
 
\subsection{\label{DISSECT}The space of geodesics hitting $H$ and disintegration of invariant measures}
Although $S^*_H M$ is not literally a cross section to the geodesic flow, inasmuch as some geodesics might not hit $H$, one might think of it
as a cross section to the geodesic flow in the set $FL(H)$. But even that is not true, because a given geodesic may intersect $H$ multiple times, and 
it is also possible that a geodesic arc or a complete geodesic lies in $S^* H$.
Roughly speaking we define the space of geodesics hitting $H$ to be  $\gfrak_H = FL(H)/\sim $ where $\sim $ is the equivalence relation of belonging to the same orbit. Since every orbit
intersects $S^*_H M$ one also has $\gfrak_H = S^*_H M/\sim$. One then
has maps $\pi: FL(H) \to \gfrak_H, \pi_1: S^*_H M \to \gfrak_H$. These maps  play an important role below in relating microlocal defect measures
on $S^*M$ to microlocal defect measures on $B^*H$. To prepare for that,
we consider disintegration of invariant measures.

The general disintegration theorem states the following: Let $(Y, \mu)$ be a probability space,  let   $\pi: Y \to X$ be a measurable map, 
and let $\nu = \pi_* \mu. $ There there exist a family of
measures $\{\mu_x\} \subset \rm{Prob}(Y)$ so that $\mu_x$ lives
on the fiber $\pi^{-1}(x)$, i.e. $\mu_x(Y \backslash \pi^{-1}(x)) = 0$ for $\nu$ a.e. $x$,
and for any measurable $f: Y \to \R_+$,
$$\int_Y f(y) d\mu(y) = \int_X \int_{\pi^{-1}(x)} f(y) d\mu_x(y) d\nu(x). $$
In our setting, $Y = FL(H), X = \gfrak_H$ and $\pi: FL(H) \to \gfrak_H$ is
the natural projection as above. 

As defined above $\gfrak_H$ is not a Hausdorff space since a geodesic may intersect $S^*_HM$ in an infinite set with an accumulation point. Moreover, the `fibers' (geodesics) have infinite measure.  For our purposes, it is possible to avoid this problem by truncation:  fix $\delta > 0$ and let $FL_{\delta}(H) = \bigcup_{|t| \leq \delta} G^t(S^*_H M). $ We then let $Y = FL_{\delta}(H), \gfrak_{\delta} = 
FL_{\delta}(H)/\sim$. This is a much simpler quotient but note  that any geodesic arcs on $H$ get collapsed to points. In particular if $H = \gamma$ is a closed geodesic, then $S^*\gamma$ is a single orbit and a single point in the quotient. We thus have a map    $\pi:S_H^*M  \to \gfrak_{\delta}$,
but it may fail to be 1-1 due to tangential geodesics. 

To remove the latter problem, we use a truncation from \cite{TZ,CGT} that punctures
out a neighborhood of the tangent directions $S^*H$ as well as in time.   
In terms
of Fermi normal coordainates $(x', x_n)$ with $H = \{x_n =0\}$, for $\delta > 0$, let $ S^*H (\delta) , \;;= \{ (x',\xi) \in S_H^*M; |\xi_n| < \delta \}$ and let $S_H^*M(\delta) = S^*_H M \backslash S^*H(\delta)$. Also let
$S^* M(H, \delta) = \{|x_n| < \delta, \; |\xi_n| > C \delta\}$ with  $C=C(H,g) > 1$ is a sufficiently large constant. We then have a map 
$$\pi_{\delta} : S^*M(H, \delta) \to  \bigcup_{|t| < \delta} G^t ( S^*_HM(\delta) )$$ 
which is 1-1 for $\delta$ sufficiently small.

Now consider a general  invariant measure $\mu$ on $S^*M$. To apply the disintegration theorem,  we first 
restrict $\mu$ to $FL(H)$, by multiplying $\mu$ by the
characteristic function ${\bf 1}_{FL(H)}$. It is equivalent to use $FL(H)$ or $\Lambda$.
Then, 
\begin{equation} \label{DIS} \left\{\begin{array}{ll}  (i) &\int_{\rm{FL}_{\delta}(H)} f d\mu  = \int_{\gfrak_{\delta}}
\left(\int_{\pi^{-1}(y)} f d\mu_y \right) d\nu(y), \\ & \\ (ii) &
\int_{S^*M(H, \delta)} f d\mu = \int_{S^*_H M(\delta)} \left(\int_{\pi^{-1}(x', \xi)} f dt \right) d\nu_{\delta}^H(x', \xi). \end{array} \right. \end{equation}
Evidently, $d\mu_y = dt$ when $d\mu$ is an invariant measure. Note
that (i) is independent of $\delta$ but the disintegration measure
$d\nu$ is not a measure on $S^*_H M$. In the integral (ii), $d \nu_{\delta}^H$  is a measure on $S^*_H M$ but depends on $\delta$. In \cite{CGT} the same measure is written in terms of Fermi-coordinates as 
\begin{equation} \label{cross-section}
  d\mu(x,\xi) = |\xi_n|^{-1} \, d\nu_{\delta}^H(x',\xi',\xi_n) \, dx_n, \quad (x,\xi) \in S^*M(H,\delta), \end{equation}
  using that $dt = |\xi_n|^{-1} \, dx_n$.
  For future reference (see Proposition \ref{weaklim}) we set 
 \begin{equation} \label{deltameasure}
 d\mu^H_{\delta}(x',\xi) :=  |\xi_n|^{-1} d\nu_{\delta}^{H}(x',\xi), \quad |\xi_n| > C \delta. \end{equation}

 The special case where $\mu = \mu_L$ (Liouville measure) is discussed in \cite{TZ} Lemma 13.

A natural question regarding (ii) is the behavior of the integrals as $\delta \to 0$.
To consider an extreme case, suppose that $d\mu$ is a periodic orbit measure $\delta_{\gamma}$ along a closed geodesic $\gamma$ of a surface
 $M$ and that
$H = \gamma$. Then the left sides of either equation are $\int_{\gamma} f ds$. On the right side of (i), $d\nu$ is a point mass at  $\gamma \in \gfrak_{\delta}$. This measure cannot be represented by the
right equation since it punctures out $\gamma \subset S^*H$.

We now formulate a condition so that the integral (i) over $\gfrak_{\delta}$ can be given by
an integral (ii) over $S^*_H M$. This is the case if $\nu$-almost every  orbit in $\gfrak_{\delta}(H)$ intersects $S^*_H M$ once. For future reference, we state this as the following

\begin{lem} \label{DIS2} If the disintegration measure $d\nu$ of  an invariant measure  $d\mu$ 
has the property that  $\nu$-almost every  orbit in $\gfrak_{\delta}(H)$ intersects $S^*_H M$ once, then   there exists a Borel measure $\nu_H$ on $S_H^*M$ with the property that 
\begin{equation} \label{DIS3} \int_{\rm{FL}_{\delta}(H)} f d\mu  = \int_{S^*_H M}
\left(\int_{\pi^{-1}(y)} f d\mu_y \right) d\nu_H(y) = \int_{S^*_H M}
\left(\int_{\pi^{-1}(y)} f dt \right) d\nu_H(y) \end{equation}

\end{lem}

\begin{proof} By deleting a set of $\nu$-measure zero of $\gfrak_{\delta}$, $\pi: S^*_H M \to \gfrak_{\delta}$ is 1-1. Hence, it admits an inverse $\pi^{-1}: \gfrak_{\delta} \to S^*_H M$. Then, $d\nu_H = (\pi^{-1})_* d\nu$ or equivalently $d\nu = \pi_* d\nu_H$.

\end{proof}

 \begin{rem} Another map is $\pi_1(x, \xi) = G^{t_1(x, \xi)}(x,\xi)$, the first impact map. If $H$ is strictly convex or concave, so that geodesics can only
 have first order contact with $H$ then the first return time to $H$ is strictly
 bounded below even for tangential directions. Hence there exists $\delta > 0$ so that each orbit in $FL_{\delta}(H)$ intersects $S^*_H M$ exactly once.
 
 \end{rem}

In Section \ref{DEFMES}, some conditions on sequences of eigenfunctions are given so that their defect measures satisfy the hypotheses of Lemma \ref{DIS2}.

\section{\label{ME} Relating matrix elements on $H$ to matrix elements on $M$} 

This section reviews the relation between matrix elements on $H$ and matrix elements on $M$. The main result (Proposition \ref{VTDECOMPa}) is repeated from \cite{TZ2}. To make this article relatively self-contained we also review the background leading to its statement and proof.

Let $(M,g)$ be a compact Riemannian manifold and let $H$ be a compact embedded $C^{\infty}$ submanifold.
  We   denote by  $U(t) = e^{i t
\sqrt{-\Delta}}$ the wave group of $(M, g)$. As is well-known,  it is a homogeneous unitary Fourier integral operator of order $0$
whose canonical relation is the graph of the homogeneous geodesic flow at time $t$; we refer to \cite{HoIII, HoIV} for background.

 We denote by $\gamma_H$ the restriction operator $\gamma_H f = f |_H: C(M) \to C(H)$ and
  by  $\gamma_H^*$   the adjoint of $\gamma_H$ with
respect to the inner product on $L^2(M, dV)$ where $dV$ is the
Riemannian volume form.  Thus,
$$\gamma_H^* f = f \delta_H, \;\; \mbox{since}\;\;
\langle \gamma_H^* f, g \rangle = \int_H f g   dS, $$ where
$ dS $ is the surface measure on $H$ induced by the
ambient Riemannian metric. The fact that $\gamma_H^*$ does not preserve
smooth functions is due to the fact that $WF_M'(\gamma_H) = N^* H$.  Thus,   $\gamma_H^* Op_H(a) \gamma_H$
is not a Fourier integral operator with a homogeneous canonical relations in the sense of \cite{HoIII} because its wave front
relation contains  $N^*H \times 0_{T^*M} \cup 0_{T^* M} \times N^* H$  (where
$0_{T^*M}$ is the zero section of $T^* M$).  For this reason we need to introduce microlocal cutoffs as in \cite{TZ2}. In the following, $\chi \in C^{\infty}_0(\R;[0,1])$ is a cutoff function with $\chi(t) = 1$ for $|t| \leq 1$ and supp $\, \chi \subset [-2,2].$

Define: 
  \begin{equation}\label{VTFORM}\left\{ \begin{array}{lll}
  V(t;a) :=U(-t) \gamma_H^* Op_H(a) \gamma_H U(t), \\ \\
  \bar{V}_{T}(a) :=  \frac{1}{T}
\int_{-\infty}^{\infty} \chi ( T^{-1} t)  \, V(t;a) \, dt,
 \end{array} \right.  \end{equation}

\begin{lem} \label{COMPLEM}  For any $a \in C_0^{\infty}(T^* H),$
\begin{equation} \label{ME1} \begin{array}{lll}  \langle Op_H(a) \phi_j |_H, \phi_j |_H \rangle_{L^2(H)} & =
& \ \langle \bar{V}_{T} (a) \phi_j, \phi_j
 \rangle_{L^2(M)},
\end{array} \end{equation}

\end{lem}

\begin{proof}
This follows from the sequence of identities,
\begin{equation} \label{ME2} \begin{array}{lll}  \langle Op_H(a) \phi_j |_H, \phi_j |_H \rangle_{L^2(H)} & =
& \langle Op_H(a) \gamma_H \phi_j, \gamma_H \phi_j \rangle_{L^2(H)} \\ & & \\
& = &  \langle \gamma_H^* Op_H (a) \gamma_H U(t) \phi_j, U(t) \phi_j
 \rangle_{L^2(M)} \\ && \\
& = &  \langle   V(t;a)  \phi_j,
 \phi_j
 \rangle_{L^2(M)} \\ && \\ & = & \langle \bar{V}_T (a) \phi_j, \phi_j
 \rangle_{L^2(M)}
\end{array} \end{equation}

\end{proof}

A detailed description of $\overline{V}_T(a)$  is given in 
Proposition \ref{VTDECOMPa} from \cite{TZ2}. There it is proved
 that, after cutting off   from the tangential  singular set $\Sigma_T
\subset T^* M \times T^* M$ and the   the conormal sets $N^*H \times 0_{T^*M}, 0_{T^* M} \times N^*H$,  $\bar{V}_{T}(a)$ becomes  a Fourier
integral operator $\VT(a)$  with canonical relation  given by
 \begin{equation}\label{GAMMADELTA}
\begin{array}{ll} WF(\VT(a)): & = \{(x, \xi, x', \xi') \in T^*M \times T^*M :  \exists t  \in (-T,T),    \\ & \\ & \exp_x t \xi = \exp_{x'} t \xi' = s \in H ,
\, \, G^t(x, \xi) |_{T_sH} = G^t(x', \xi') |_{T_s H}, \;\; |\xi| =
|\xi'|\}.
\end{array} \end{equation}

\subsection{\label{CUTOFFSECT} Good cutoffs}

 In view of (\ref{measure1}) and Lemma \ref{open}, we consider the disjoint, closed sets
 $$ K_1 := S^*M \setminus \Lambda, \quad K_{2} := \overline{\Lambda_{R,\epsilon}}$$
 where $R= R(\delta)$ and $\epsilon = \epsilon(\delta)$ are chosen as in (\ref{measure1}). Thus, by the $C^{\infty}$ Urysohn lemma, there exists a cutoff $\chi_{R,\ep} \in C^{\infty}(S^*M; [0,1])$ with 
 \begin{equation} \label{urysohn}
\chi_{R,\ep}(x,\xi) = \begin{cases} 1 & \, (x,\xi) \in K_{2} \nonumber \\ 0 & \,  (x,\xi) \in K_1. \end{cases} \end{equation}
We abuse notation somewhat and denote the postive homogeneous degree zero extension of $\chi_{R,\ep}$ to $T^*M-0$ also by $\chi_{R,\ep}$ and the corresponding pseudodifferential operator by $\chi_{R,\ep}(x,D_x) \in Op(S^{0}(T^*M-0)).$

At this point, as  in \cite{TZ} we introduce some further cutoff operators  supported away from glancing and conormal directions to $H.$ 
For fixed $\epsilon >0,$  let $\chi^{(tan)}_{\epsilon}(x, D) = Op(\chi_{\epsilon}^{(tan)}) \in Op(S^0_{cl}(T^*M)),$ with 
homogeneous symbol $\chi^{(tan)}_{\epsilon}(x,\xi)$ supported in an $\epsilon$-aperture conic neighbourhood of $T^*H \subset T^*M$ with $\chi^{(tan)}_{\epsilon} \equiv 1$ in an $\frac{\epsilon}{2}$-aperture subcone. The second cutoff operator $\chi^{(n)}_{\epsilon}(x,D) = Op(\chi_{\epsilon}^{(n)}) \in Op(S^0_{cl}(T^*M))$ has its homogeneous symbol  $\chi^{(n)}_{\epsilon}(x,\xi)$ supported in an $\epsilon$-conic neighbourhood of $N^*H$ with $\chi^{(n)}_{\epsilon} \equiv 1$ in an $\frac{\epsilon}{2}$ subcone. Both $\chi_{\epsilon}^{(tan)}$ and $\chi_{\epsilon}^{(n)}$ have spatial support in the set where $|x_n| < \epsilon$ (see \cite{TZ} (5.1) and (5.2)). 
To simplify notation, define the total cutoff operator
\begin{equation} \label{CUTOFF}
\chi_{\epsilon}(x,D) := \chi^{(tan)}_{\epsilon}(x,D) + \chi^{(n)}_{\epsilon}(x,D).
\end{equation}

 \subsection{Cutoff of $\gamma_H^* Op_H(a)
 \gamma_H$ and its time average}
 
 We define 
\begin{equation} \label{>ep}
(\gamma_H^* Op_H(a) \gamma_H)_{\geq \epsilon} = (I - \chi_{\frac{\epsilon}{2}}) \gamma_H^* Op_H(a) \gamma_H (I - \chi_{\epsilon}), \end{equation}
and
\begin{equation} \label{<ep} (\gamma_H^* Op_H(a)
 \gamma_H)_{\leq \epsilon} =  \chi_{2 \epsilon} \gamma_H^* Op_H(a) \gamma_H \chi_{\epsilon}. \end{equation}

By a standard wave front calculation, it follows that  \begin{equation} \label{op decomp}
\gamma_H^* Op_H(a)
 \gamma_H = (\gamma_H^* Op_H(a)
 \gamma_H) )_{\geq \epsilon} +  (\gamma_H^* Op_H(a)
 \gamma_H) )_{\leq \epsilon}  + K_{\epsilon},\end{equation}
 where, $\langle K_{\epsilon} \phi_j, \phi_j \rangle_{L^2(M)} = {\mathcal O}(\lambda_j^{-\infty}).$
We then   define
  \begin{equation} \label{VtFORMep}
  V_{\epsilon} (t;a) :=U(-t) (\gamma_H^* Op_H(a) \gamma_H)_{\geq \epsilon} U(t),  \end{equation}
and
 \begin{equation}\label{VTFORMep}   \VT(a) :=  \frac{1}{T}  \int_{-\infty}^{\infty} \chi ( T^{-1} t)  \, V_{\epsilon}(t;a) \, dt.    \end{equation}
 



The next proposition provides  a detailed description of $\VT(a) $ as  a Fourier integral operator
with local canonical graph away from its fold set and computes its  principal symbol.

  \begin{prop}\label{VTDECOMPa}  Fix $T,\epsilon >0$ and let
  $a \in S^{0}_{cl}(T^*H)$ with    $a_H(s, \xi) = a(s, \xi |_H) \in S^0(T^*_H M) ). $  Then $\VT(a)$ is a Fourier
integral operator with local canonical graph,
and possesses the decomposition
  $$\VT(a) =  P_{T,\epsilon}(a) + F_{T,\epsilon}(a) + R_{T,\epsilon}(a),$$
where, (i) $P_{T,\epsilon}(a) \in Op_{cl}(S^0(T^*M))$ is a
pseudo-differential operator of order zero with principal symbol
 \begin{equation} \label{aTFORM}  a_{T,\epsilon}(x,\xi)
 :=\sigma(P_{T,\epsilon}(a))(x,\xi) = \frac{1}{T}  \sum_{j \in {\mathbb Z}}  (1-\chi_{\epsilon})  (\pi^* \gamma^{-1} a_H)(G^{t_j(x,\xi)}(x,\xi))  \, \chi (T^{-1}  t_j(x,\xi))  \end{equation}
  where,
 $t_j(x,\xi) \in C^{\infty}(T^*M)$ are the impact times of the geodesic $\exp_{x}(t \xi)$ with $H$, and $\gamma$ is defined by (\ref{gammaDEF}).

(ii) $F_{T,\epsilon}(a)$ is a
 Fourier integral operator of order zero  with canonical relation
 $\Gamma_{T,\epsilon}$. \begin{equation} \label{FIOsum}
F_{T,\epsilon}(a) = \sum_{j=1}^{N_{T,\epsilon}}
F_{T,\epsilon}^{(j)}(a),\end{equation} where the
$F_{T,\epsilon}^{(j)}(a); j=1,..., N_{T,\epsilon}$ are zeroth-order
homogeneous Fourier integral operators with
$$ WF'(F_{T,\epsilon}^{(j)}(a)) = \text{graph}  (\rcal_j) \cap \Gamma_{T,\epsilon},$$
and symbol
$$\sigma(F_{T,\epsilon}^{(j)})(x,\xi) =  \frac{1}{T}   (\gamma^{-1} a_H)( G^{t_j(x, \xi)}(x, \xi)) \, \chi(T^{-1} t_j(x,\xi)) \ |dx d\xi|^{\frac{1}{2}}.$$

(iii) $R_{T,\epsilon}(a)$ is a smoothing operator.

\end{prop}

The proof of Propositon \ref{VTDECOMPa}  goes roughly as follows: we decompose  $\VT(a)$ into  a pseudo-differential
and a Fourier integral part according to the dichotomy that 
$(x, \xi, x', \xi')$ in (\ref{GAMMADELTA}) satisfy either
\begin{equation} \label{possibilities}
\ \begin{array}{ll}
 (i) \,\, G^t(x,\xi) = G^t (x',\xi'), \; \mbox{or}\\ \\
 (ii) \,\, G^{t}(x',\xi') = r_{H} G^{t}(x,\xi), \end{array}  \end{equation}
 where $r_H$ is the reflection map of $T^* H$  in (\ref{rHDEF}). Thus, 
\begin{equation} \label{WFVT} WF(\VT(a)) = \Delta_{T^* M \times T^* M} \cup \Gamma_T. \end{equation}
The pseudo-differential part $P_{T, \epsilon}$  of $\VT(a)$ is its microlocalization to (i) and
the Fourier integral part $F_{T, \epsilon}$ is its microlocalization to (ii). For further details, we refer the reader to Proposition 7 in \cite{TZ}.

 \section{Proof of Proposition \ref{FTPROP} and Theorem \ref{MDMPROP}}
   By \eqref{ME1}-\eqref{ME2} and by 
 Proposition \ref{VTDECOMPa},   the weak*  limits of the restricted
 matrix elements are those of
 
 \begin{align} \label{decomp}
 \langle
\VT(a)  \phi_j,  \phi_j
 \rangle_{L^2(M)}  &=  \langle
P_{T, \epsilon}  \phi_j,  \phi_j
 \rangle_{L^2(M)}  +   \langle
F_{T, \epsilon} \phi_j,  \phi_j
 \rangle_{L^2(M)} + \langle R_{T, \epsilon} \phi_j, \phi_j \rangle_{L^2(M)}.  \end{align}\
 
It is clear that $ \langle R_{T, \epsilon} \phi_j, \phi_j \rangle_{L^2(M)} \to 0$
for the entire sequence of eigenfunctions. We now argue that the $F_{T, \epsilon}$ term is negligible for a density one subsequence. 





%
%

\medskip

\subsection{\label{FTSECT} Removing the $F_{T, \epsilon}$ term and proof of Proposition \ref{FTPROP}}
We now consider the Fourier integral matrix elements
  $   \langle
F_{T, \epsilon} \phi_j,  \phi_j
 \rangle_{L^2(M)} $.
 
 \begin{lem} \label{FTOUT} Suppose that $H$ is an asymmetric hypersurface. Then for any fixed $T >0,\epsilon >0$ 
there exists a subsequence \ $ \scal_F(T,\epsilon)$ of the eigenfunctions of density one such that
  $$    \langle
F_{T, \epsilon} \phi_j,  \phi_j
 \rangle_{L^2(M)}   \to 0, \quad j \in \scal_F(T,\epsilon).$$
 \end{lem}

 \begin{proof} 
 It suffices to show that  \begin{equation} \label{FIO part} \limsup_{\lambda \to \infty}
  \frac{1}{N(\lambda)} \sum_{j: \lambda_j \leq \lambda} \left|  \langle F_{T,\epsilon}(a)  \phi_j,  \phi_j
 \rangle_{L^2(M)} \right|^2 = 0.  \end{equation} 
 The proof of this Lemma is essentially  identical to that in \cite{TZ2},
 since it did not use ergodicity of the geodesic flow. Hence we only sketch it
 for the sake of comleteness.
 First, we note that for any $R>0,$  we clearly have
 $$ \langle F_{T,\epsilon}(a) \phi_j, \phi_j \rangle_{L^2(M)} = \langle F_{R,T,\epsilon}(a) \phi_j, \phi_j \rangle_{L^2(M)},$$
 where $F_{R,T,\epsilon}(a) := \frac{1}{2R} \int_{-R}^R U(r)^* F_{T,\epsilon}(a) U(r) \, dr.$ Then, the Weyl sum in (\ref{FIO part} by the corresponding Weyl sum with $F_{T,\epsilon}$ replaced with $F_{R,T,\epsilon}$.
 To prove (\ref{FIO part}) we first use the Schwartz inequality
\begin{equation} \label{SIE} \frac{1}{N(\lambda)} \sum_{j: \lambda_j \leq \lambda} \left|
\langle F_{R,T,\epsilon}(a)  \phi_j,  \phi_j
 \rangle_{L^2(M)} \right|^2 \leq \frac{1}{N(\lambda)} \sum_{j: \lambda_j \leq \lambda}   \langle F_{R,T,\epsilon}(a)^* F_{R,T,\epsilon}(a)  \phi_j,  \phi_j
 \rangle_{L^2(M)} \end{equation}
 to bound the variance sum by a trace. We then use the local Weyl
 law for Fourier integral operators associated to local canonical graphs,
\begin{equation} \label{oldWeyl}
\frac{1}{N(\lambda)} \sum_{j: \lambda_j \leq \lambda} \langle F
\phi_{\lambda_j}, \phi_{\lambda_j} \rangle \to \int_{S \Gamma_F
\cap \Delta_{T^*M} } \sigma_{\Delta} (F) d\mu_L, \end{equation}
where $\Gamma_F$  is the canonical relation of $F$,  $S \Gamma_F$
is the set of vectors of norm one, and  $S \Gamma_F \cap
\Delta_{T^*M} $  is its intersection  with the diagonal of $T^*M
\times T^*M$. Also,  $ \sigma_{\Delta} (F)$ is the (scalar) symbol
in this set and $d\mu_L$ is Liouville measure. Thus, if $\Gamma_F$
is a local canonical graph, the  right side is zero unless the
intersection has dimension $m = \dim M$. The microlocal asymmetry
condition is precisely that the intersection has measure zero.  Then, as in \cite{TZ} Lemma 3, one gets that under the asymmerry condition on $H,$
\begin{equation} \label{FIO part 2} \limsup_{\lambda \to \infty}
  \frac{1}{N(\lambda)} \sum_{j: \lambda_j \leq \lambda} \left|  \langle F_{R,T,\epsilon}(a)  \phi_j,  \phi_j
 \rangle_{L^2(M)} \right|^2 = o_{T,\epsilon}(1)  \end{equation} 

\noindent as $R \to \infty.$ Since the LHS in  (\ref{FIO part 2}) equals the LHS in (\ref{FIO part}) and the latter is independent of $R$, letting $R \to \infty$ in (\ref{FIO part 2}) completes the proof of Proposition \ref{FTPROP}\end{proof}

As in the introduction, for fixed $\epsilon \in (0,1),$ let $\chi_{\epsilon}^H \in C^{\infty}_0(B^*H)$ be a cutoff with supp $\chi_{\epsilon}^H \subset \{ (s,\sigma) \in B^*H:  | |\sigma|_s - 1 | < \epsilon \}.$  
Proposition \ref{FTPROP} then follows from Lemma \ref{COMPLEM},  Proposition \ref{VTDECOMPa} and Lemma \ref{FTOUT}.

\subsection{\label{DEFMES}Proof of Theorem \ref{MDMPROP}}\label{cs} 

To complete the proof of Theorem \ref{MDMPROP}  we `disintegrate' each  microlocal defect measure $d\mu$ of the sequence $\scal_F$ in the sense of Section \ref{DISSECT}. As discussed there, we first have to localize
$d\mu$ to $FL(H)$, and we further localize it to $FL_{\delta}(H)$ by
multiplying by its characteristic function. If we apply the disintegration
theorem, we obtain a measure on $\gfrak_{\delta}$.  Proposition \ref{MDMPROP} asserts that there exists a sequence of density one
so that each of its microlocal defect measures can be disintegrated to
a measure on $S^*_H M$.  To prove this, we show that there exists
a subsequence of density one for which Lemma \ref{DIS2} holds.
That is, only a subsequence of density zero can charge $S^*H$.

Let  $\scal' \subset \scal_F$  be a subsequence corresponding to a global defect measure $d\mu_{\scal'}^M$ on $S^*M$. To simplify notation we simply write $d\mu = d\mu_{\scal'}^M$ below (and similiarily, we write $d\mu^H$ for $d\mu_{\scal'}^H$).

 \begin{prop} \label{weaklim} There exists a density one set  $\tilde{\scal} \subset \scal_F$ such that, for any defect measure $\mu$ arising from a subsequence  $ \{ \phi_{j_k} \}$  with $j_k \in \scal' \subset \tilde{\scal}$, the corresponding Borel measures $d\mu_{\delta}^H$ on $S^*M(H,\delta)$ defined  in  (\ref{deltameasure}) converge weakly as $\delta \to 0$ to a Borel measure $d\mu^H$ on $S_H^*M.$  
 \end{prop}
\begin{proof}
As suggested above, a sequence failing to have this property must blow
up along $H$. An example would be restrictions of Gaussian beams along a closed geodesic $\gamma$ to the geodesic, or restrictions of whispering gallery modes of a convex domain to its boundary. We now prove that there exists a subsquence of density one so
that \begin{equation} \label{UB}  \| \phi_{h_j}^{H} \|_{L^2(H)}  = O(1). \end{equation}
In fact, the next Lemma proves more:
\begin{lem}\label{uniformweyl}
Fix $\delta_0>0$ small and let $H_{\tau} = \{ x_n = \tau \}$ with $|\tau| < \delta_0.$ Then, there is a density-one sequence $\tilde{\scal}$ such that 
$$\sup_{|\tau| \leq \delta_0} \| \phi_{h_j}^{H_{\tau}} \|_{L^2(H_{\tau})}  = O(1), \quad j \in \tilde{\scal},$$
as $h_j \to 0$.
\end{lem}
\begin{proof} Let $\{\tau_n\}_{n=1}^{\infty}$ be a countable set in $[0,1]$ and let  $H_{\tau_n}$ be the corresponding  sequence of hypersurfaces.

 By  the pointwise Weyl law, $\frac{1}{N(h)} \sum_{h_j \geq h} \| \phi_{h_j} \|_{L^2(H_{\tau_n})} \sim_{h \to 0} c_n | H_{\tau_n}|$ and consequently for each $n$, there exists a density-one subset $\tilde{S}_n$ such that for $j \in \tilde{S}_n,$ \begin{equation} \label{bounded}
\| \phi_{h_j}^{H_{\tau}} \|_{L^2(H_{\tau_n}) }  \leq C, \quad j \in \tilde{\scal}_n.
\end{equation}
where $C>0$ is independent of $n.$ Since $\tilde{\scal} = \bigcap_{n\geq 1} \tilde{\scal}_n$ is also of density-one, it follows from (\ref{bounded}) that
\begin{equation} \label{bounded2}
 \| \phi_{h_j}^{H_{\tau}} \|_{L^2(H_{\tau_n}) }  \leq C, \quad j \in \tilde{\scal}, \quad n=1,2,3.... \end{equation}
 
 Now consider a general $H_{\tau}$. We pick $\{\tau_n\}$ to be dense in $[0,1]$ and to contain $0$. In particular, \eqref{UB} holds for $j \in \tilde{\scal}$.
 To prove that it holds for all $\tau \in [0, \delta_0]$ for some $\delta_0 > 0$ we argue by contradiction.
For  any $\tau$  and any $\epsilon >0$, there exists $H_{\tau_n}$ with $d(H_{\tau_n}, H_\tau) = \inf_{(x,y) \in H_{\tau_n} \times H_{\tau}} d(x,y) < \epsilon.$ We then consider the functions
 $$\rho (\tau) = \limsup_{j \to \infty} \| \phi_{h_j} \|_{L^2(H_{\tau})}. $$
 Since $\tau \to  \| \phi_{h_j} \|_{L^2(H_{\tau})}$ is a continuous function, $\rho(\tau)$ is lower  semi-continuous.  If $\rho $ is not bounded on any
 interval $[0, \delta]$,   then there exists a sequence
 $\{\hat{\tau}_k\}_{k=1}^{\infty}$ with $\tau_k \to 0$ (disjoint from $\{\tau_n\}$) and so that $ \| \phi_{h_j}^{H_{\hat{\tau}_k}} \|_{L^2(H_{\hat{\tau}_k}) } > k$.
 Since $\rho$ is lower semi-continuous, each superlevel set $\{\rho > k\}$ is open and non-empty. But this contradicts the fact that
 $\rho \leq C$ on  the dense set $\{\tau_n\}$.

\end{proof}

We now show that the disintegration measure of any defect measure arising from a subsequence in $\scal_F \cap \tilde{S}$ is a measure
on $S^*_H M$, i.e. the measures $d\mu_{\delta}^H$ have a weak limit.

Suppose $ 0 \leq a \in C^{\infty}(S_H^*M).$  Then, since $\xi_n^2 = 1 - |\xi'|_{x'}^2$ for $(x',\xi) \in S_H^*M,$  stereograph projection maps $\pi_H^{\pm}: S^{\pm,*}_{H}M  \setminus S^*H \to \mathring{B}^*H$ given by $\pi_H^{\pm}(x',\xi',\pm \sqrt{1- |\xi'|^2}) = (x',\xi')$  are diffeomorphisms.  Consequently, there exist $a^{\pm} \in C^{\infty}(\mathring{B^*H})$ with  $a^{\pm}\circ \pi_H^{\pm} = a |_{S_{H}^{\pm,*}M }$ and  in terms of Fermi coordinates, 
$a^{\pm}(x',\xi') = a (x', \xi', \pm \sqrt{1-|\xi'|_{x'}^2}), \,\, (x',\xi') \in \mathring{B}^*H.$   Next,  we decompose Riemann measure and write $dx = d\sigma_{x_n=\tau} \, d\tau$ where $d\sigma_{x_n= \tau}$ is hypersurface measure on $H_{\tau} = \{ x_n = \tau \}.$ Then, in view of (\ref{cross-section}), one can  write
\begin{align} \label{wk}
\int_{S_H^*M} a \, d\mu^H_{\delta}  &= \int_{S_{H}^{+,*}M} a d\mu^H_{\delta} + \int_{S_{H}^{-,*}M} a d\mu^H_{\delta} \nonumber \\
&\leq  \limsup_{h \to 0} \, \Big(  \sup_{|\tau|< \delta_0}  \langle a^+(x',hD')  (1-\chi_{C\delta})(\sqrt{I +h^2 \Delta_{H_\tau} }) \,   \phi_h^{H_\tau}, \phi_h^{H_\tau} \rangle_{L^2(H_\tau)} \Big) \nonumber \\
& + \limsup_{h \to 0} \, \Big( \sup_{|\tau| < \delta_0}  \langle a^-(x',hD')  (1-\chi_{C\delta})(\sqrt{I +h^2 \Delta_{H_\tau} }) \,    \phi^{H_\tau}_h, \phi^{H_\tau}_h \rangle_{L^2(H_\tau)}  \Big)
\nonumber \\
&\leq \limsup_{h \to 0}   \, \sup_{|\tau| < \delta_0} \langle a(x',hD') \phi_h^{H_\tau}, \phi_h^{H_\tau}  \rangle_{L^2(H_{\tau})} = O(1)\|a\|_{L^\infty}, \end{align}
by $L^2$-boundedness and Lemma \ref{uniformweyl}.


 Since the $O(1)$ bound on the RHS of (\ref{wk}) is uniform in $\delta$ and
$\int_{S_H^*M} a d\mu_{H,\delta}$ is monotone non-decreasing as $\delta \to 0,$ it follows that for $0 \leq a \in C^{\infty}(S_H^*M)$, the  $\lim_{\delta \to 0} \int_{S_H^*M} a d\mu_{H,\delta}$ exists.
 A  similar argument applies to $a \in C^{\infty}$ with $a \leq 0.$ In  general, for $a \in C^{0}(S_H^{*}M)$ we make the decomposition $a = a_{+} - a_{-}$ and mollify $a_{\pm}$ by considering $a_{\pm, \beta}:= \chi_{\beta}  * a_{\pm} \in C^{\infty}$ with $ 0 \leq \chi_{\beta} \in C^{\infty}$ an approximation of the identity. One then applies the above argument to $a_{\beta,\pm}$ separately and takes the $\beta \to 0$ limit at the end.

 Finally, setting 
$$ d\mu^H(a):= \lim_{\delta \to 0} \int_{S_H^*M} a \, d\mu^H_{\delta},$$
it is clear that $d\mu^H$ is linear, non-negative and from (\ref{wk}) satisfies $ |d\mu^H(a)| \leq C \|  a \|_{L^\infty(S_H^*M)}$ and is consequently  a measure on $S_H^*M.$

\end{proof} 

 \begin{rem} We note that the eigenfunction subsequences with defect measures $d\mu$ satisfying the conditions in Proposition \ref{weaklim} are precisely the ones for which the map $\pi: S_H^*M \to \gfrak_{\delta}$  in Lemma \ref{DIS2} of section \ref{DISSECT}  is  $\nu$ almost everywhere 1-1. 

\end{rem}

 In view of Proposition \ref{FTPROP}, to  complete the proof of Theorem
\ref{MDMPROP}, we must compute the integral of $\sigma(P_{T, \epsilon})$ against an invariant measure  $\mu$ using the disintegration decomposition of $d\mu$ in (\ref{DIS}) (ii) and Proposition \ref{weaklim}.  First, we recall  from (\ref{DIS}) (ii) and (\ref{deltameasure}) that for $\delta >0$ sufficiently small,
\begin{equation} \label{disdecomp0}
d\mu |_{S^*M(H,\delta)} = dt d\nu_{\delta}^H |_{S^*M(H,\delta)} = dx_n \, \, d\mu_{\delta}^H(x',\xi), \quad (x',\xi) \in S_H^*M (\delta).\end{equation}

Provided the defect measure $d\mu$ corresponds to a density-one eigenfunction subsequence in Proposition \ref{weaklim}, one can take the weak limit in $(\ref{disdecomp0})$ as $\delta \to 0$. The result is that
\begin{equation} \label{disdecompupshot}
d\mu(x,\xi)  =  d x_n \, d\mu^H(x',\xi), \quad (x,\xi) \in S^*M,
\end{equation}
where $d\mu^H$ is a Borel measure on $S_H^*M.$

Provided one chooses $T = T(\epsilon)$ small enough so that there is only one term in \eqref{aTFORM}, it then follows
by \eqref{aTFORM} and (\ref{disdecompupshot}) that
\begin{equation}\begin{array}{lll}  \label{lastcomp} \int_{S^*M}   a_{T,\epsilon}(x,\xi) d\mu & = & \frac{1}{T}
\int_{S^*M}   (1-\chi_{\epsilon})  (\pi^*\gamma^{-1} a_H)(G^{t_1(x,\xi)}(x,\xi)) \, d\mu \\&&\\
& = &   \frac{1}{T} \int_0^T \chi(\frac{t}{T}) dt 
\int_{S_H^*M}   (1-\chi_{\epsilon})(x',\xi) \, \pi^*\gamma^{-1}(x', \xi) \, a_H(x',\xi)  \, \pi^*\gamma(x',\xi)\, d\mu^H(x',\xi) \\&&\\
& = &   
\int_{S_H^*M}   (1-\chi_{\epsilon})  a_H (x', \xi)  \, d\mu^H(x',\xi).
\end{array} \end{equation} 
In the penultimate line of (\ref{lastcomp}) we have used that for $(x',\xi) \in S_H^*M, \,\, |\xi_n|  = \pi^* \gamma (x',\xi)$ so that the $\pi^*\gamma^{-1}$ factor in the symbol of $P_{T,\epsilon}$ gets cancelled by the $|\xi_n| = \pi^*\gamma$ factor in the numerator coming from the disintegration of $d\mu.$ 

From Proposition \ref{FTPROP} it then follows that for any $\epsilon >0$, there exists a density-one sequence $\tilde{\scal}(\epsilon) \subset \scal_F(T(\epsilon),\epsilon)$ such that
\begin{equation} \label{upshotfinal}
  \langle Op_H(a(1-\chi_{\epsilon}^H) ) \phi_{j_k} |_H, \phi_{j_k} |_H \rangle_{L^2(H)}  \sim_{k \to \infty} \int_{S_H^*M}   (1-\chi_{\epsilon})  a_H (x', \xi)   d\mu^H(x',\xi), \quad j_k \in \tilde{\scal}(\epsilon).\end{equation}
Finally, choose a sequence $\epsilon_n, \, n=1,2,3,...$ with $\epsilon_n \to 0^+$  in (\ref{upshotfinal}) and set
$\tilde{\scal}:= \cap_{n \geq 1} \tilde{S}(\epsilon_n).$
Theorem \ref{MDMPROP} then  follows by taking the  $\epsilon_n \to 0$ limit  in (\ref{upshotfinal}), with the result that for all $a \in S^0(H),$
$$ \langle Op_H(a) \phi_{j_k} |_H, \phi_{j_k} |_H \rangle_{L^2(H)}  \sim_{k \to \infty} \int_{S_H^*M}    a_H (x', \xi)   d\mu^H(x',\xi), \quad j_k \in \tilde{\scal}.$$

\qed





\section{\label{MMSECT} Mass and microsuppport: Proof of Theorem \ref{MASSMICRO}}

We consider the  space $$\acal_H = \bigcup_{T, \epsilon > 0}  \{P_{T, \epsilon}(a): a \in S^0(B^*H)\}$$ of ``cross-sectional pseudo-differential operators'' operators acting on $L^2(M)$ and let $$a_{T, \epsilon} = \sigma_{P_{T, \epsilon}}. $$

 \begin{defin} We define the {\it cross-sectional symbol space} $S^0 \acal_H$
  to be
 the space of  zeroth-order symbols $a_{T, \epsilon}$ of $P_{T, \epsilon} \in \acal$.  \end{defin}

As in \cite{Ge}, define  the wave front set   $WF(\scal)$  of
obstructions to microlocal compactness of $\scal$ as follows:

\begin{defin} We define the semi-classical   wave front set of a sequence $\scal =\{u_n\}$  with respect to $\acal_H$ such as
$$WF_{\acal_H} (\scal) = \bigcap_{A \in \acal_H, A \scal \; \rm{compact}} \{\sigma_A = 0\},$$
where the intersection runs over all $A \in \acal_H$ such that $A u_n$
is relatively compact in $L^2$. 
\end{defin}

By a microlocal defect measure $\mu$
of $\scal$ we mean a probability measure on $S^* M$ obtained as a weak*
limit of the functionals $\rho_j(A) = \langle A u_j, u_j \rangle$. In the case that $\scal$ has a unique microlocal defect measure (quantum limit), a well-known
result equates the wave front set with the support of the microlocal
defect measure: $WF(\scal)  = \rm{Supp} \mu$; see \cite{Ge} for background.
We define the relative analogue using the subspace $\acal_H$:

 \begin{lem} If $\scal = \{\phi_j\}$ is a sequence satisfying $||\phi_j |_H||_{L^2(H)}
= o(1)$,   then any microlocal defect measure   (quantum limit measure)  $\mu$ of $\scal$ on $S^*M$ satisfies 
$$ \rm{supp} \mu \subset \bigcap_{a, \epsilon, T} \{\sigma_{P_{T, \epsilon}(a)} =0\}.$$
\end{lem}
 \bigskip
 
 \begin{proof} 

It is obvious that if  $||\phi_j |_H||_{L^2(H)}
= o(1)$  then $||Op_{h_j}(a) \phi_j||_{L^2(H)} \to 0$  for all $a \in C^{\infty}_c(T^* H)$. Hence, all microlocal defect measures of
the sequence $\langle Op_{h_j}(a) \phi_j, \phi_j \rangle$ on $B^* H$
must vanish. 

Proposition \ref{MDMPROP}  relates matrix elements on $H$ to matrix elements on $M$.  If $||\phi_j |_H||_{L^2(H)}
= o(1)$ then
 $$0 =\lim_{k \to \infty} \langle P_{T, \epsilon}(a) \phi_{j_k}, \phi_{j_k} \rangle
 = \int_{S^*M} \sigma_{P_{T, \epsilon}(a)}  d \mu. $$
 
 \end{proof}
 
 \begin{rem} Note that $d\mu$ is the microlocal defect measure of $\scal$
 on $S^*M$. It does not need to equal $d\mu_{\scal}^M$ since the latter is the defect measure only relative to the subspace $\acal_H$. What the Lemma
 asserts is that both measures must have the same integrals with respect to symbols of operators in $\acal_H$. \end{rem}

We now want to show that $ \{\sigma_{P_{T, \epsilon}(a)} =0\}$ has measure
zero and that a microlocal defect measure supported in a set of measure
zero must come from a zero-density subsequence.

\begin{lem} $ \bigcap_{A \in \acal_H} \{\sigma_A = 0\} \subset S^*M \backslash FL(H)$. That is, $ \bigcap_{a} \{\sigma_{P_{T, \epsilon}(a)} =0\}$ is the
complement of the flowout $FL(H)$.  \end{lem}

\begin{proof}
We denote by $\sigma \acal_H$ the set of all possible symbols of $P_{T, \epsilon} \in \acal_H$. By Lemma \ref{MEPLEM},
 if  $a_H > 0$ then $\sigma_{P_{T, \epsilon}} (x, \xi) > 0$ if $G^t(x, \xi)$ intersects $S^*_H M$. Hence the set  $\bigcap_{a} \{\sigma_{P_{T, \epsilon}(a)} =0\}$ cannot contain any points for $T$ small (depending on $\epsilon$ and $\epsilon > 0$. 
  As a result, 
 the zero set can only contain points $(x, \xi)$ for which
 the orbit never hits $H$.

 \end{proof}

\subsection{Spectral  projections in $\hcal_{\scal}$}

We have been considering microlocal defect measures (quantum limits) of
$\langle P_{T, \epsilon}(a) \phi_{j}, \phi_j \rangle$.  But we may also consider microlocal defect measures of the normalized traces   
\begin{equation} \label{rhoscal} \rho_{\scal, \lambda} (A) :  = \frac{1}{N(\lambda, \scal)} \rm{Tr}
A \Pi_{\scal, \lambda}, \;\; A \in \acal_H, \end{equation}
where if $\scal =\{\phi_{j_k}\}$ then
$$\Pi_{\scal,\lambda} f = \sum_{j: \lambda_{j_k} \leq \lambda} \langle f, \phi_{j_k} \rangle \phi_{j_k}. $$
These are states on the space $\acal_H$. 

\begin{lem} Let $\mu_{\scal}$ be a microlocal defect measure for
the functionals $\rho_{\scal, \lambda}$. Then $a_{T, \epsilon} \mu_{\scal} = 0$
for all symbols in $\sigma \acal_H$.  \end{lem}

\begin{proof}

 The argument above for individual eigenfunctions is also  true for the microlocal
lift of the projector $\Pi_{\scal, \lambda}$. Pick $\epsilon, T$ so that
the complement has measure $< \delta$, the putative density of $\scal$.

Let $d \Phi_j$ be a the postive microlocal
lift of $\phi_j$, i.e $\int_{S^*M} a d \Phi_j = \langle Op(a) \phi_j, \phi_j \rangle$
where $Op(a)$ is a positive quantization (for example, a Friedrichs quantization).   $\scal$ and its density are  independent of $T, \epsilon$. But the limit $\mu_{\scal}$ of
$$\hat{\rho}_{\scal, \lambda} : = \frac{1}{N(\lambda, \scal)} \sum_{j: \lambda_j \leq \lambda, \lambda_j \in \scal}
d \Phi_{j} $$
must also satisfy
$$a_{T, \epsilon} \mu_{\scal} = 0, \;\; .$$
\end{proof}

\begin{cor} The  defect measures $\mu_{\scal} $ of the trace funtionals $\rho_{\scal, \lambda} $ are supported in a the complement of $FL(H)$ in $S^*M$, a  set of Liouville measure zero.
\end{cor}

\begin{proof}

The limit is a $G^t$ invariant probability measure. Hence $\mu_{\scal}$
must vanish on $FL(H)$. Thus, $\mu_{\scal}$ is supported on a closed invariant
set of Liouville measure zero, i.e. $\scal$ is a positive
sequence which `concentrates' on a closed invariant set
$\Gamma$ of $\mu_L$-measure zero, where $\mu_L$ is Liouville measure.

\end{proof}

To get a contradiction, we need to show that if $\scal$ has positive
density, then the microlocal defect measures cannot all be supported
in a set of measure zero.

\begin{prop} Suppose that $\Gamma$ is a closed invariant set
of $\mu_L$-measure 0, and that $\scal$ is a subsequence all of
whose microlocal defect measures are supported in $\Gamma$. Then
$D^*(\scal) =0$.  
\end{prop}

\begin{proof}

We argue by contradiction and show that if  $\scal$ has positive
density, then the `maximal' microlocal defect measure cannot be supported in 
a set of $\mu_L$ measure zero. This maximal measure comes from 
the spectral projections onto the sequence $\scal$.

If $\scal$ has positive density, then there exists $A > 0$ so that
$$ \limsup_{\lambda \to \infty} \frac{N(\lambda) }{N(\lambda, \scal)}  \leq A. $$

Let $V$ be a conic neighborhood of $\Gamma$ and let $\chi_V$ be a conic cutoff
to $V$. Then for any $a$,
\begin{equation} \label{PUTCHI} \lim_{\lambda \to \infty} \hat{ \rho}_{\scal, \lambda}(Op(a)) =\lim_{\lambda \to \infty}  \hat{\rho}_{\scal, \lambda}(Op(\chi_V a)). \end{equation}

Let $$\hat{\rho}_{ \lambda} : = \frac{1}{N(\lambda)} \sum_{j: \lambda_j \leq \lambda}
d \Phi_{j}. $$ Recall the  local Weyl law
 \begin{equation} \label{LWL}\hat{\rho}_{\lambda}(A) = 
  \frac{1}{N(\lambda)}
 \sum_{j: \lambda_j \leq \lambda} \langle A \phi_j,  \phi_j
 \rangle_{L^2(M)}  \to \int_{S^*M} \sigma_A d\mu_L  \end{equation} on $M$.

For $\lambda$ sufficiently large, by \eqref{PUTCHI},
$$\begin{array}{lll}\limsup_{\lambda} \frac{1}{N(\lambda, \scal)} \sum_{j: \lambda_j \leq \lambda, \lambda_j \in \scal}
\rho_j(Op(a)) & = &\limsup_{\lambda}  \frac{N(\lambda) }{N(\lambda, \scal)} \frac{1}{N(\lambda)} \sum_{j: \lambda_j \leq \lambda, \lambda_j \in \scal}
\rho_j(Op(\chi_V a))\\&&\\
& \leq &  A \limsup_{\lambda}  \frac{1}{N(\lambda)} \sum_{j: \lambda_j \leq \lambda}
\rho_j(Op(\chi_V a)) \\&&\\ &=&  A \limsup_{\lambda}  \; \hat{\rho}_{\lambda}(Op(\chi a)) \leq \mu_L(V). \end{array}$$
If $\mu_L(\Gamma) = 0$ the the right side is $\leq \epsilon$ if $V$ is
an $\epsilon$-neighborhood. It follows that $ \lim_{\lambda \to \infty} \hat{ \rho}_{\scal, \lambda}(Op(a)) =0$ for all $a$, which is absurd. This contradiction completes the proof.

\end{proof}

\section{\label{PFSECT}Proof of Theorem \ref{mainthm1} }
In this section we prove the quantitative refinement of Theorem \ref{MASSMICRO} stated in Theorem \ref{mainthm1}.

\begin{proof} 
To prove Theorem \ref{mainthm1} we study integrals $\int_H f |\phi_j|^2 dS$
or more general matrix elements $\langle Op_h(a) \gamma_H^* \phi_j,
\gamma_H^* \phi_j \rangle_{L^2(H)}$. In order to prove that $H$ is good
for a density one sequence of eigenfunctions, it suffices to show that
 the matrix elements to do not tend to zero for at least one symbol
$a$.
 Since 
$$ \langle (\gamma_H^* Op_H(a)
 \gamma_H) )_{\geq \epsilon} \phi_j |_H, \phi_j |_H \rangle_{L^2(H)} = \langle \VT(a)  \phi_j,  \phi_j
 \rangle_{L^2(M)} $$ it follows from (\ref{decomp}) that
 
 \begin{align} \label{key1}
 \limsup_{\lambda \to \infty} \frac{1}{N(\lambda)} \sum_{j: \lambda_j \leq \lambda} \Big|  \langle (\gamma_H^* Op_H(a)
 \gamma_H) )_{\geq \epsilon} \phi_j , \phi_j \rangle_{L^2(M)}   -  \langle  [ P_{T,\epsilon}(a)  + F_{T,\epsilon}(a)] \phi_j,  \phi_j
 \rangle_{L^2(M)} \Big|^2  = 0. \end{align}

\subsection{Contribution of the pseudo-differential term $P_{T, \epsilon}(a)$.}

In view of (\ref{FIO part}), it follows from (\ref{key1}) that
\begin{equation} \label{key2}
\limsup_{\lambda \to \infty} \frac{1}{N(\lambda)} \sum_{j: \lambda_j \leq \lambda} \left|  \langle (\gamma_H^* Op_H(a)
 \gamma_H) )_{\geq \epsilon} \phi_j , \phi_j \rangle_{L^2(M)}   - \langle P_{T,\epsilon}(a)  \phi_j,  \phi_j
 \rangle_{L^2(M)}  \right|^2 = 0. \end{equation}
 
  Since we are free to choose the non-negative symbol $a$, we henceforth put 
 $a(s,\sigma):= 1$  and simply write 
\begin{equation} \label{PT1} P^1_{T,\epsilon}:= P_{T,\epsilon}(1) \end{equation}  with
 \begin{equation} \label{symbol}
 \sigma(P^1_{T,\epsilon})(x,\xi) = \frac{1}{T}  \sum_{j \in {\mathbb Z}}   (1-\chi_{\epsilon}) \pi_H^*(\gamma^{-1})(G^{t_j(x,\xi)}(x,\xi))  \, \chi (T^{-1}  t_j(x,\xi))  \end{equation}

\subsubsection{Microlocal ellipticity of $P^1_{T, \epsilon}$}

We now observe   that for fixed $T > 0, \epsilon >0,$ 
 $P_{T,\ep}^1 $ is microlocally elliptic  on the support of  $\chi_{T, \epsilon}.$

\begin{lem} \label{MEPLEM} We have \begin{equation} \label{SIGMAPLB} \sigma(P^1_{T,\ep})(x,\xi) \geq \frac{1}{T}, \quad (x,\xi) \in  \rm{supp}  \chi_{T,\ep}. \end{equation}
\end{lem}

\begin{proof}


The symbol of $P_{T, \epsilon}^1$ is

\begin{equation} \label{aTFORM2}   \sigma(P_{T,\epsilon}^1)(x,\xi)  
= \frac{1}{T}  \sum_{j \in {\mathbb Z}}  (1-\chi_{\epsilon}) \pi_H^* (\gamma^{-1})(G^{t_j(x,\xi)}(x,\xi))  \, \chi (T^{-1}  t_j(x,\xi))  \end{equation}
  where,
 $t_j(x,\xi) \in C^{\infty}(T^*M)$ are the impact times of the geodesic $\exp_{x}(t \xi)$ with $H$, where $\gamma$ is defined by (\ref{gammaDEF}).
 
 By definition of the cutoff \eqref{urysohn}, it follows that for $(x,\xi) \in \text{supp} \, \chi_{T,\ep},$ the hitting time   $|t_1(x,\xi)| < T$ and $(1-\chi_\ep) ( G^{t_1(x,\xi)}(x,\xi)) = 1.$ Consequently,
 
$$ \frac{1}{T}  \sum_{j \in {\mathbb Z}}  (1-\chi_{\epsilon}) \pi_H^* (\gamma^{-1})(G^{t_j(x,\xi)}(x,\xi))  \, \chi (T^{-1}  t_j(x,\xi))  \geq  \frac{1}{T} \pi_H^*(\gamma^{-1})(G^{t_1(x,\xi)}(x,\xi))  \geq \frac{1}{T}$$
since $ \pi_H^*(\gamma^{-1})(s,\eta) \geq 1$ for any $(s,\eta) \in S_H^*M.$
 
 

\end{proof}

By the pointwise local Weyl law, for any $a \in S^0(T^*H),$

$$\limsup_{\lambda \to \infty} \frac{1}{N(\lambda)} \sum_{j: \lambda_j \leq \lambda} \left|  \langle (\gamma_H^* Op_H(a)
 \gamma_H) )_{\leq \epsilon} \phi_j , \phi_j \rangle_{L^2(M)} \right|^2 = O(\ep)$$ 
 and so, from (\ref{key2}) if follows that for any fixed $T>0,$
 
 \begin{equation} \label{key3}
 \limsup_{\lambda \to \infty} \frac{1}{N(\lambda)} \sum_{j: \lambda_j \leq \lambda} \left|  \, \| \phi_j^H \|_{L^2(H)}^2   - \langle P^1_{T,\epsilon} \phi_j,  \phi_j
 \rangle_{L^2(M)}  \,  \right|^2 =  O(\ep). \end{equation}

 We note that in (\ref{key3}), one is free to choose the time-average parameter $T$. In \cite{TZ} we take $T \to \infty$ in order to apply the mean ergodic theorem, but here we will {\em not} take $T \to \infty;$ rather, here $T$  will be a fixed constant to be specified later on. Consequently,  taking $\liminf_{\ep \to 0}$ of both sides of (\ref{key3}), it follows that  there is a density-one subset $S \subset \{1,...,\lambda \}$ such that

\begin{equation} \label{liminf}
\liminf_{\ep \to 0} | \,  \| \phi_j^H \|^2_{L^2(H)}  -  \langle P^1_{T,\epsilon}  \phi_j,  \phi_j
 \rangle_{L^2(M)}  \, | = o(1), \quad \lambda_j \to \infty, \,  j \in S. \end{equation}\

As a consequence of (\ref{liminf}) it suffices to estimate $ \liminf_{\ep \to 0} \langle P^1_{T,\epsilon}  \phi_j,  \phi_j
 \rangle_{L^2(M)}$ from below. To do this, we will need the microlocal ellipticity result in Lemma \ref{MEPLEM} combined with the following   lemma on a priori, microlocal eigenfunction  mass estimates near $S_H^*M$ under the full-measure flowout assumption.

  \subsubsection{Microlocal eigenfunction mass estimates}

In this section we prove that for all $\delta >0$ there exists a positive   lower bound $C_{\delta} > 0$ 
on a density $\geq 1 - \delta$ subsequence for the matrix
elements of $\langle P_{T, \epsilon} \phi_{j_k}, \phi_{j_k} \rangle$
where $\epsilon = \epsilon(\delta)$ we be taken sufficiently small.
  Here, $P_{T, \epsilon} $ corresponds to a positive
symbol on $B^*H$ and as above we take it to equal $1$. Then,
$P_{T, \epsilon} =  \chi_{T, \ep(\delta)}(x,D_x)$ \eqref{urysohn}.)

 \begin{lem} \label{mass} Fix $T >0.$ Then, for any $0 < \delta <1,$ there exists a subsequence $\scal(\delta)$ of density greater that $1-\delta$ such that with $\ep = \ep(\delta)>0$ sufficiently mall, there exists $C_{\delta}>0,$
 $$ \langle \chi_{T, \ep(\delta)}(x,D_x) \phi_j, \phi_j \rangle_{L^2(M)} \geq C_{\delta}, \quad j \in \scal(\delta).$$
 \end{lem}

\begin{proof} Throughtout  $T >0$ will be fixed and so dependence of constants on $T$ will be suppressed.  W let $R>0$ be an independent parameter that we will choose sufficiently large (see (\ref{small volume}) below).  Letting $(\chi_{T,\ep(\delta)})_R := \frac{1}{R} \int_{0}^R U(-t) \chi_{T,\ep(\delta)} U(t) dt,$ it follows that
\begin{equation} \label{egorov1}
\langle \chi_{T,\ep(\delta)} \phi_j, \phi_j \rangle_{L^2(M)} = \langle (\chi_{T,\ep(\delta)})_R \, \phi_j, \phi_j \rangle_{L^2(M)} + O_{T}(\lambda_j^{-1}).\end{equation} \
In view of (\ref{measure1}) and since $\sigma ( (\chi_{T,\ep(\delta)})_R )(x,\xi) = \frac{1}{R} \int_0^R \chi_{T,\ep(\delta)}(G^t(x,\xi)) dt$ can find $R = R(\delta)$ and enlarge $\ep(\delta)$ to  $ \ep'(\delta) > \epsilon(\delta)$ so that  \begin{equation} \label{small volume}
 \mu_L ( \, \text{supp}(1- \chi_{R(\delta), \ep'(\delta)} )  \, ) = O(\delta)\end{equation}
 and there exists $C(\delta)>0$ with
 
\begin{equation} \label{elliptic1}
 \sigma( (\chi_{T,\ep(\delta)})_{R(\delta)} )(x,\xi) \geq 2 C(\delta)>0, \quad (x,\xi) \in \text{supp} \, \chi_{R(\delta),\ep'(\delta)}.\end{equation}

 To simplify the writing, in the following we sometimes suppress the dependence of $R, \ep$ and $\ep'$ on $\delta$. 

 

 By \eqref{elliptic1},  $(\chi_{T,\ep})_R $  is microlocally elliptic on supp $ \chi_{R,\ep'} $ and so by the sharp Garding inequality (cf. \cite[Theorem 18.1.14]{HoIII} or \cite[Theorem 6.1, p. 20]{Tay}) applied to the
 operator $\chi_{R, \epsilon}(x, D) - 2 C(\delta) \, \chi_{R, \epsilon'}(x,D)$,
 
 \begin{equation} \label{elliptic2}
 \langle (\chi_{T,\ep})_R (x,D_x) \phi_j, \phi_j \rangle_{L^2(M)} \geq 2 C(\delta) \langle \chi_{R,\ep'} \phi_j, \phi_j \rangle_{L^2(M)}  + O_\delta(\lambda_j^{-1})\end{equation}
 We recall that  the sharp Garding inequality states: If   $p \in S^{0}$ with $\Re p \geq 0$, then  there exists a constant $C_0 > 0$ so that   $\Re \langle p(x,D) u,u \rangle \geq - C_0 ||u||^2_{H^{-1/2}}.$  
 

One can write
\begin{align} \label{egorov2}
\langle \chi_{R,\ep} \phi_j, \phi_j \rangle = \beta(\lambda_j;R,\ep) - C_{\delta} \langle (I-\chi_{R,\ep'}) \phi_j, \phi_j \rangle_{L^2(M)}+ O_{\delta}(\lambda_j^{-1}), &
\end{align}
where $$\beta(\lambda_j;R,\ep):=   \, \langle (\chi_{T,\ep})_R \chi_{R,\ep'} \phi_j , \phi_j \rangle + \langle  (\chi_{T,\ep})_R (I- \chi_{R,\ep'}) \phi_j, \phi_j \rangle  + C_\delta \langle (I-\chi_{R,\ep'}) \phi_j, \phi_j \rangle. $$
The point of isolating the $\beta(\lambda_j,R,\ep)$ term on the RHS of (\ref{egorov2}) is that, as we now show,  this term is uniformly bounded from below as $\lambda_j \to \infty.$ The  $C_{\delta}(1- \chi_{R,\ep'}) $-term is added in the definition of $\beta(\lambda_j,R,\ep)$ to get a globally elliptic operator.  

More precisely, from (\ref{elliptic2}) \begin{align} \label{beta1}
\beta(\lambda_j;R,\ep) & \geq C_\delta \langle \chi_{R,\ep'} \phi_j, \phi_j \rangle + \langle  (\chi_{T,\ep})_R (I- \chi_{R,\ep'}) \phi_j, \phi_j \rangle + C_\delta \langle (I-\chi_{R,\ep'}) \phi_j, \phi_j \rangle. \end{align}

Using the fact that $(\chi_{T,\ep})_R \,  (1-\chi_{R,\ep'}(x,\xi)) \geq 0,$ it follows by application of sharp Garding in the second term on the RHS of (\ref{beta1}) that
\begin{equation} \label{beta2}
 \beta(\lambda_j;R,\ep)  \geq C_\delta  \, \langle \chi_{R,\ep'} \phi_j, \phi_j \rangle + C_\delta \langle (I-\chi_{R,\ep'}) \phi_j, \phi_j \rangle + O(\lambda_j^{-1}) \geq C_\delta + O(\lambda_j^{-1}),\end{equation}
since $\| \phi_j \|_{L^2}^2 =1.$ 

To estimate the matrix value $\langle \chi_{R,\ep} \phi_j, \phi_j \rangle,$ the term involving $C_\delta (1- \chi_{R,\ep'})$ is subtracted out in (\ref{egorov2}), but in the variance sum this term gives a small contribution since  $ \mu_L ( \text{supp} \, (1-\chi_{R,\ep'})$ is small.  More precisely, since $|\langle (I-\chi_{R,\ep'}) \phi_j, \phi_j \rangle |^2
\leq \langle (I-\chi_{R,\ep'})^2 \phi_j, \phi_j \rangle $ it follows from the   local Weyl law  that
$$ \limsup_{\lambda \to \infty} \frac{1}{N(\lambda)} \sum_{\lambda_j \leq \lambda} | \, \langle (I-\chi_{R,\ep'}) \phi_j, \phi_j \rangle |^2 \leq \int_{S^*M} (1 -\chi_{R,\ep'})^2=  O(\delta),$$
from which it follows that
$$\limsup_{\lambda \to \infty} \frac{1}{N(\lambda)} \sum_{\lambda_j \leq \lambda} | \, \langle \chi_{R,\ep} \phi_j, \phi_j \rangle - \beta(\lambda_j;R,\ep) \, |^2   = O(C_\delta \, \delta ).$$

By Chebyshev's inequality,
$$D^* ( \{j;  | \langle \chi_{R,\ep} \phi_j, \phi_j \rangle -  \beta(\lambda_j;R,\ep| 
\geq \frac{C_\delta}{2} \} )  $$ 
$$ = O\Big(\frac{2}{C_\delta} C_\delta \, \delta \Big) = O(\delta),$$
and consequently,
\begin{equation} \label{DCAL1}  D^* ( \{j;  | \langle \chi_{R,\ep} \phi_j, \phi_j \rangle -  \beta(\lambda_j;R,\ep| 
\leq \frac{C_\delta}{2} \} ) \geq 1 - C \delta, \end{equation}

where $C>0$ is a constant independent of $\delta >0.$
For  eigenfunctions $\phi_j$ satisfying $ | \langle \chi_{R,\ep} \phi_j, \phi_j \rangle -  \beta(\lambda_j;T,\ep| 
\leq \frac{C_\delta}{2}$  it folllows from the lower bound  $\beta(\lambda_j;R,\ep)\geq C_\delta + O(\lambda_j^{-1})$ in (\ref{beta2}) that for $\lambda_j$ sufficiently large, 
$$ \langle \chi_{R,\ep} \phi_j, \phi_j \rangle \geq \frac{C_{\delta} }{2}>0.$$ 

Since then
\begin{equation} \label{DCAL2} D^* (\{j;   \langle \chi_{R,\ep} \phi_j, \phi_j \rangle \geq  \frac{C_\delta}{2} \} )  \geq 1 - C \delta,\end{equation}

that finishes the proof of Lemma \ref{mass}.


\end{proof}


 We are now in a position to prove lower bounds for  $\langle P^1_{T,\ep(\delta)} \phi_j, \phi_j \rangle_{L^2(M)}$ where $j \in S(\delta)$ with ${\mathcal D}(S(\delta)) \geq 1- \delta.$
 To do this, we use the  sharp Garding inequality  yet again.
Recalling \eqref{PT1}, we have $\sigma(P^1_{T, \ep})
\geq \sigma (P^1_{T,\ep} \chi_{T,\ep})$ and since  $P^1_{T,\ep} \in Op(S^0)$ and $P^1_{T,\ep} \, \chi_{T,\ep} \in Op(S^0),$ it  follows from the sharp Garding inequality that

$$\langle P^1_{T, \ep} \phi_j, \phi_j \rangle_{L^2(M)} \geq  \langle P^1_{T,\ep} \chi_{T,\ep} \phi_j, \phi_j \rangle_{L^2(M)}  - C_{\ep} \| \phi_j \|_{H^{-\frac{1}{2}}(M)}^2,$$
and so,

\begin{equation} \label{elliptic4}
\langle P^1_{T, \ep} \phi_j, \phi_j \rangle_{L^2(M)} \geq  \langle P^1_{T,\ep} \chi_{T,\ep} \phi_j, \phi_j \rangle_{L^2(M)} + O_{\ep}(\lambda_j^{-1}). \end{equation}\

In view of (\ref{elliptic4}), it is enough to bound the matrix elements $\langle P^1_{T,\ep} \chi_{T,\ep} \phi_j, \phi_j \rangle_{L^2(M)}$ from below.  Combining  the microlocal mass estimate in Lemma \ref{mass} and the microlocal ellipticity result in Lemma \ref{MEPLEM}, it follows by sharp Garding that with $\epsilon = \epsilon(\delta)$ sufficiently small,

\begin{equation} \label{elliptic5}
\langle P^1_{T, \ep(\delta)} \, \chi_{T,\ep(\delta)}\, \phi_j, \phi_j \rangle_{L^2(M)} \geq \frac{1}{T} \langle \chi_{T,\ep(\delta)} \phi_j, \phi_j \rangle_{L^2(M)}  - C_{\delta} \lambda_j^{-1} \geq \frac{C_{\delta}}{2T},  \quad j \in \scal(\delta), \,\, \lambda_j \geq \lambda(\delta). \end{equation}\
 Consequently, from (\ref{elliptic4}) and (\ref{elliptic5}),
\begin{equation} \label{elliptic6b}
\langle P^1_{T, \ep(\delta)} \phi_j, \phi_j \rangle_{L^2(M)} \geq  \frac{C_{\delta}}{2T}  + O_{\delta}(\lambda_j^{-1}), \quad j \in \scal(\delta). \end{equation}\

From (\ref{liminf}) it follows that after possibly shrinking $\ep(\delta)$ further, 
\begin{equation} \label{liminf2}
| \,  \| \phi_j^H \|^2_{L^2(H)}  -  \langle P^1_{T,\epsilon(\delta)}  \phi_j,  \phi_j
 \rangle_{L^2(M)}  \, | \leq \frac{C_{\delta}}{4T}, \quad \lambda_j  \geq \lambda(\delta),\,\,  j \in \scal. \end{equation}\

 Now, restricting to  $j \in \scal(\delta) \cap \scal$ in (\ref{liminf2}) and using (\ref{elliptic6b}) 
 one gets
 $$ \| \phi_j^{H} \|^2_{L^2(H)} \geq \frac{ C_{\delta}}{4T}, \quad \lambda_j \geq \lambda(\delta), \,\, j \in \scal \cap \scal(\delta).$$
 
This completes the proof of Theorem  \ref{mainthm1}.


 
 

 \end{proof}

  \section{Examples of hypersurfaces with $\mu_L(\rm{FL}(H)) = 1.$} \label{examples}

\subsection{Simple convex surfaces of revolution}

Let $(M,g)$ be a strictly-convex surface of revolution with metric
$g = d\theta^2 + f(\theta) d\phi^2$ where $0 \leq \theta \leq L$ and $\phi \in [0,2\pi]$ and $f \in C^{\infty}([0,L],\R)$ with $f(\theta) >0,$
$$ f'(\theta_0) = 0, \quad f''(\theta_0) <0 $$
and
$$ f'(\theta) \neq 0 \quad \text{for} \,\,  \theta \neq \theta_0.$$

The Hamiltonian 
$$H(\theta,\phi,\xi_{\theta}, \xi_{\phi}) = \xi_{\theta}^2 + f^{-1}(\theta) \xi_{\phi}^2$$
is Liouville completely integrable with integral in involution
$P((\theta,\phi,\xi_{\theta}, \xi_{\phi}) = \xi_{\phi}.$

The ``equator" of the surface is the periodic geodesic
$$\gamma_0 := \{ \theta = \theta_0, \,\, 0 \leq \phi \leq 2\pi \}.$$
The moment map (restricted to $S^*M$) is
$${\mathcal P}:= (1, \xi_{\phi}): S^*M \to \R^2$$
If ${\mathcal B}_{reg}$ denotes the regular values of the moment map, ${\mathcal P}^{-1}(b)$ consists of (two) Lagrangian tori \cite{TZ4} (with slight abuse of notation we denote them both by $\Lambda_{b}\subset S^*M$). If $S^*M_{reg} := \cup_{b \in {\mathcal B}_{reg}} \Lambda_b$ then 
$$ \mu_L ( S^*M_{reg} ) = 1.$$
 The projections $\pi(\Lambda_b)$ with $b \in {\mathcal B}_{reg}$ of the invariant tori are the equatorial bands 
\begin{equation} \label{band}
\pi(\Lambda_b) = \{ (\theta, \phi);  \theta \in [\theta_0 - r_1(b), \theta_0 + r_1(b)], \, 0 \leq \phi \leq 2\pi \} \end{equation}
where $f(r_j(b)) =  b, \,\, j=1,2.$

On $S^*M_{reg}$ there exist action-angle variables $(\theta, I)$ in terms of which the Hamiltionian $H = H(I)$ and so the Hamltion equations are solvable by quadrature. Explicitly, the action variables in this case are \cite{TZ4}
$$I_1(b) = b, \quad I_2(b) = \frac{1}{\pi} \int_{r_1(b)}^{r_2(b)} \Big( 1- \frac{b^2}{f^2(\theta)} \Big)^{\frac{1}{2}} \, d\theta.$$

Under the twist assumption
\begin{equation} \label{twist}
\nabla_I \omega(I) \neq 0, \quad \omega(I)  := \nabla_I H(I) \end{equation}
on the metric, there exists a family of irrational tori (which we denote by $\cup_{b \in \Q^c} \Lambda_b$ such that the geodesic flow on each such torus $\Lambda_b$ is dense. That is, for any $(\theta, I) \in \Lambda_b $ with $b \in \Q^c$ we have that 
\begin{equation} \label{dense}
\overline{ \bigcup_{t \in \R} \gamma(t; I,\theta) } = \Lambda_b,\quad b \in \Q^c. \end{equation}
Consequently, for  the projected geodesic,
\begin{equation} \label{denseproj}
\overline{ \bigcup_{t \in \R} \pi \circ \gamma(t; I,\theta) } = \pi(\Lambda_b),\quad b \in \Q^c. \end{equation}

Now suppose $H \subset M$ is an simple, closed curve with
$$H \cap \gamma_0 \neq \emptyset.$$
  Then, in view of (\ref{denseproj}) and the band stucture in (\ref{band}) it follows that for any geodesic $\gamma(t; \theta,I)$ on $\Lambda_b$ with $b \in \Q^c,$ the projection $\pi \circ \gamma(t;\theta,I)$ must intersect $H$ for some $t \in \R$ and so,
$$\bigcup_{t \in \R} \gamma(t;\theta,I) \cap S_H^*M \neq \emptyset.$$

Then since $\Lambda_b$ is $G^t$-invariant, $ \bigcup_{b \in \Q^c} \Lambda_b \subset \rm{FL}(H)$ and so,
$$\mu_L( \rm{FL}(H)) \geq \mu_L ( \bigcup_{b \in \Q^c} \Lambda_b) = \mu_L( S^*M_{reg}) = 1.$$

Since trivially $\mu_L (\rm{FL}(H)) \leq 1$ it follows that $\mu_L( \rm{FL}(H)) =1.$ \\

To summarize: under the twist condition (\ref{twist}), for any  simple closed curve $H$ with 
$H \cap \gamma_0 \neq \emptyset,$
we have that $\mu_L (\rm{FL}(H)) =1.$\\

It is well-known \cite{Bl} that both oblong $(a<b)$ and oblate $(a>b)$ ellipsoids $\frac{x^2}{a^2} + \frac{y^2}{a^2} + \frac{z^2}{b^2} = 1$ satisfy the twist condition, whereas the sphere  does not. Indeed,  in the latter case $\gamma_0 = \{ (x,y,z) \in S^2; z =0 \}$ and so the condition $H \cap \gamma_0 \neq \emptyset$ is clearly not sufficient since for any closed curve $H$ with diam $H < $ diam $\gamma_0$ there exist a postiive measure of $2\pi$-periodic great circles that do not intersect $H.$

\subsection{Liouville metrics on tori}
A Liouville torus $(M,g)$ is a topological two-torus with metric
$$ g = [U_1(x_1) - U_2(x_2)] ( dx_1^2 + dx_2^2), \quad x = (x_1,x_2) \in [0,1] \times [0,1].$$
Here, $U_j$ are smooth Morse functions with period $1$ and are required ro satisfy
$U_1(x_1) - U_2(x_2) >0$ for all $x \in [0,1] \times [0,1].$ Moreover, we assume that $U_1$ and $U_2$ each have one maximum and minimum and that these critical points are all distinct.
The associated geodesic flow is generated by the Hamiltonian
$$H(x,\xi) = [ U_1(x_1) - U_2(x_2)]^{-1} ( \xi_1^2 + \xi_2^2)$$
and the integral in involution is
$$P(x,\xi) = \frac{U_2(x_2)}{ [ U_1(x_1) - U_2(x_2)]  } \xi_1^2 + \frac{U_1(x_1)}{[ U_1(x_1) - U_2(x_2)]} \xi_2^2.$$
The restricted moment map is then
${\mathcal P} = (1,P): S^*M \to \R^2.$  Then \cite{TZ4}, the singluar leaves of the Lagrangian foliation consist of two horizontal periodic geodesics $\gamma_h^1$ and $\gamma_h^2$ along with two vertical ones $\gamma_v^1$ and $\gamma_v^2$. The associated bands (projections onto $M$ of invariant Lagrangian tori) come in two horizontal (resp. vertical) families containing the projected geodesics $\pi \circ \gamma_h^{1,2}$ (resp. $\pi \circ \gamma_v^{1,2}).$ (see \cite{TZ4} for further details).\\

A similar argument to the one above for revolution surfaces, shows that under a twist assumption on $g,$ for any simple closed curve $H \subset M$ satisfying
$$ H \cap \gamma_h^{1,2} \neq \emptyset, \quad H \cap \gamma_v^{1,2} \neq \emptyset,$$
we have $\mu_L (\rm{FL}(H)) =1.$


\begin{thebibliography}{Shn2}
 
 \bibitem[AB]{AB}  J.C.  Alvarez Paiva and G.  Berck, What is wrong with the Hausdorff measure in Finsler spaces. Adv. Math. 204 (2006), no. 2, 647--663.
 
 \bibitem[AP]{AP} J. C. Alvarez Paiva and E. Fernandes,  Gelfand transforms and Crofton formulas. Selecta Math. (N.S.) 13 (2007), no. 3, 369--390.
 
 \bibitem[Ber]{Ber} B. Berndtsson, Restrictions of plurisubharmonic functions
to submanifolds (preprint, 2016).

\bibitem[Bl]{Bl} P. Bleher, Distribution of Energy Levels of a Quantum Free Particle on a Surface of Revolution.     Duke Math. J.
 74 (1994), no. 1,  45-93.
 
 





\bibitem[BR11]{BR2} J. Bourgain and Z. Rudnick, On the nodal sets of toral eigenfunctions, Invent. Math. 185  (2011), no. 1, 199-237.

\bibitem[BR12]{BR}  J.  Bourgain and Z.  Rudnick, 
Restriction of toral eigenfunctions to hypersurfaces and nodal sets,  Geom. Funct. Anal. 22 (2012), no. 4, 878-937 (arXiv:1105.0018).


\bibitem[CGT17]{CGT} Y. Canzani, J. Galkowski and J. A. Toth, Averages of eigenfunctions over hypersurfaces (arXiv:1705.09595).

\bibitem[CT]{CT} Y. Canzani and J.A. Toth, Intersection bounds for nodal sets of Laplace eigenfunctions, to appear in  Conference Proceedings on Algebraic and Analytic Microlocal Analysis (Northwestern, 2015), M. Hitrik, D. Tamarkin, B. Tsygan, and S. Zelditch, eds. Springer.


\bibitem[CTZ]{CTZ} H. Christianson, J. A. Toth and S. Zelditch, Quantum ergodic restriction theorems for Cauchy data,  Math. Res. Lett. 20 (2013), no. 3, 465-475.





\bibitem[DF]{DF} H. Donnelly and C. Fefferman, Nodal sets of eigenfunctions on
Riemannian manifolds, Invent. Math. 93 (1988), 161-183, MR1039348, Zbl 0784.31006.

\bibitem[DZ]{DZ}  S. Dyatlov and M. Zworski, Quantum ergodicity for restrictions to hypersurfaces. Nonlinearity 26 (2013), no. 1, 35-52.

\bibitem[ET]{ET} L. El-Hajj and J. A. Toth, Intersection bounds for nodal sets of planar Neumann eigenfunctions with interior analytic curves. J. Differential Geom. 100 (2015), no. 1, 1-53.

\bibitem[G]{G} J. Galkowski, The $L^2$-behaviour of eigenfunctions near the glancing set. Comm. P.D.E. 41 (2016), 1619-1648.


\bibitem[Ge]{Ge} P. Gerard, 
Microlocal defect measures. 
Comm. Partial Differential Equations 16 (1991), no. 11, 1761-1794. 


\bibitem[GRS]{GRS} A. Ghosh, A. Reznikov, and P.  Sarnak,
Nodal domains of Maass forms I, Geom. Funct. Anal. 23 (2013), no. 5, 1515–1568 (arXiv:  1207.6625).


















\bibitem[HoIII]{HoIII}  L. H\"ormander, {\it The analysis of linear partial differential operators. III. Pseudo-differential operators.}  Classics in Mathematics. Springer, Berlin, 2007.

\bibitem[HoIV]{HoIV}  L. H\"ormander, {\it  The analysis of linear partial differential operators.}
 IV. Fourier integral operators. Reprint of the 1994 edition. Classics in Mathematics. Springer-Verlag, Berlin, 2009. 

\bibitem[Ho2]{Ho2} L. H\"ormander,   Notions of convexity. Reprint of the 1994 edition. Modern Birkhäuser Classics. Birkhäuser Boston, Inc., Boston, MA, 2007.






\bibitem[JJ]{JJ} J. Jung,  Sharp bounds for the intersection of nodal lines with certain curves. J. Eur. Math. Soc. (JEMS) 16 (2014), no. 2, 273-288.

\bibitem[JJZ]{JJZ} J.  Jung and S. Zelditch,  Number of nodal domains and singular points of eigenfunctions of negatively curved surfaces with an isometric involution. J. Differential Geom. 102 (2016), no. 1, 37-66. 

\bibitem[K]{K} V. Kaloshin, 
A geometric proof of the existence of Whitney stratifications. 
Mosc. Math. J. 5 (2005), no. 1, 125-133. 








\bibitem[L]{L} G. Lebeau, The complex Poisson kernel on a compact analytic Riemannian manifold, to appear  (2013). 

\bibitem[LG]{LG}  P. Lelong and L. Gruman, {\it Entire functions of several complex variables.}  Grundlehren der Mathematischen Wissenschaften 282. Springer-Verlag, Berlin, 1986



\bibitem[Lin]{Lin} F-H.  Lin, Nodal sets of solutions of elliptic and parabolic equations. Comm. Pure Appl. Math. 44 (1991), no. 3, 287-308.













\bibitem[St]{St} M. Stenzel, On the analytic continuation of the Poisson kernel, Manuscripta Math. 144 (2014), no. 1-2, 253-276.


\bibitem[Tay]{Tay} M.E. Taylor, {\it Pseudodifferential Operarors} Princeton University Press, Princeton, NJ., 1981.


\bibitem[TZ]{TZ} J. A. Toth and S. Zelditch,   Counting
Nodal Lines Which Touch the Boundary of an Analytic Domain, Jour.
Diff. Geom. 81 (2009), 649- 686 (arXiv:0710.0101).


\bibitem[TZ13]{TZ2}  J.A. Toth and S. Zelditch,  Quantum ergodic restriction theorems: manifolds without boundary. Geom. Funct. Anal. 23 (2013), no. 2, 715-775 (arXiv:1104.4531).


\bibitem[TZ12]{TZ3}  J.A. Toth  and S. Zelditch,  Quantum ergodic restriction theorems: I. Interior hypersurfaces in domains with ergodic billiards Ann. H. Poincare 13 (2012) 599-670. 

\bibitem[TZ03]{TZ4} J.A. Toth and S. Zelditch, Norms of modes and quasimodes revisited. Proc. of AMS (Harmonic Analysis at Mount Holyoke) (2003), 435-458.

\bibitem[Ze07]{Zerg} S. Zelditch, Complex zeros of real ergodic eigenfunctions. Invent. Math. 167 (2007), no. 2,
419 - 443.


\bibitem[ZPl]{ZPl} S. Zelditch, Pluri-potential  theory on Grauert tubes of  real analytic
Riemannian manifolds, I {\it Spectral geometry}, 299-339, 
Proc. Sympos. Pure Math., 84, Amer. Math. Soc., Providence, RI, 2012. 
58J50 (32U99 32V99 35P20)  (arXiv:1107.0463).


\bibitem[Zint]{Zint} 
 S. Zelditch,  Ergodicity and intersections of nodal sets and geodesics on real analytic surfaces. J. Differential Geom. 96 (2014), no. 2, 305-351.
 
 
\bibitem[Ze16]{ZSt} S. Zelditch, 
Measure of nodal sets of analytic Steklov eigenfunctions, to appear in Math. Res. Letts (arXiv:1403.0647).


\bibitem[Zw]{Zw} M. Zworski, {\it Semiclassical analysis}, Graduate Studies in Mathematics, 138. American Mathematical Society, Providence, RI (2012).



\end{thebibliography}
\end{document}